\documentclass[12pt,reqno,a4paper]{amsart}
\usepackage[ansinew]{inputenc}
\usepackage{amsfonts}
\usepackage{latexsym}
\usepackage{mathrsfs}
\usepackage{amsmath}
\usepackage{amssymb}
\usepackage{eepic}
\usepackage{color}
\usepackage[margin=3.6cm]{geometry}
\usepackage{hyperref}
\hypersetup{
    colorlinks=true,
    linkcolor=blue,
    citecolor=red,
}

\newcommand{\veps}{\varepsilon}
\newcommand{\R}{\mathbb{R}}

\newcommand{\C}{\mathbb{C}}
\newcommand{\N}{\mathbb{N}}
\newcommand{\Z}{\mathbb{Z}}

\newtheorem{lettertheorem}{Theorem}

\newtheorem{letterlemma}[lettertheorem]{Lemma}

\newtheorem{defin}{Definition}[section]

\newtheorem{theorem}[defin]{Theorem}
\newtheorem{exa}{Example}
\newenvironment{example}{\begin{exa}\rm}{\end{exa}}\newtheorem{lemma}[defin]{Lemma}
\newtheorem{corollary}[defin]{Corollary}
\newtheorem{rem}{Remark}
\newenvironment{remark}{\begin{rem}\rm}{\end{rem}}

\numberwithin{equation}{section}

\makeatletter
\renewcommand{\ps@myheadings}{%
\renewcommand{\@evenhead}%
{{\rm\thepage}\hfil{\sc J.~Heittokangas, J.~Wang, Z.~T.~Wen and H.~Yu}\hfil}%
\renewcommand{\@oddhead}%
{\hfil{{\sc The $\varphi$-order for the $q$-differences}\hfil{\rm\thepage}}}%
\renewcommand{\@evenfoot}{}%
\renewcommand{\@oddfoot}{}%
}\makeatother 

\title{Meromorphic functions of finite $\varphi$-order and linear $q$-difference equations}

\author[Heittokangas]{J.~Heittokangas}
\address[Heittokangas\\Yu]{
Department of Physics and Mathematics, University of Eastern Finland, P.O.~Box 111, 80101 Joensuu, Finland}

\author[Wang]{J.~Wang}
\address[Wang]{School of Mathematical Sciences, Fudan University, Shanghai  200433, P. R. China}

\author[Wen]{Z.~T.~Wen}
\address[Wen]{Department of Mathematics, Shantou University, Shantou 515063, Guangdong, P. R. China}

\author[Yu]{Hui Yu$^*$}
\thanks{$^*$ Corresponding author.}

\email{janne.heittokangas@uef.fi \\majwang@fudan.edu.cn\\zhtwen@stu.edu.cn\\huiy@uef.fi}

\date{\today}

\begin{document}
\maketitle

\begin{abstract}
   The $\varphi$-order was introduced in 2009 for meromorphic functions in the unit disc, and was used as a growth indicator for solutions of linear differential equations.  In this paper,  the properties of meromorphic functions in the complex plane are investigated in terms of the $\varphi$-order, which measures the growth of functions  between the classical order and the  logarithmic order. Several results on value distribution of meromorphic functions are discussed by using the $\varphi$-order and the $\varphi$-exponent of convergence. Instead of linear differential equations, the applications in the complex plane lie in linear $q$-difference equations.


\end{abstract}

\renewcommand{\thefootnote}{ }
\footnote{}

%%%%%%%%%%%%%%%%%%%%%%%%%%%%%%%%%%%%%%
% SECTION 1
%%%%%%%%%%%%%%%%%%%%%%%%%%%%%%%%%%%%%%
\section{Introduction}

Let $\varphi:(R_0,\infty)\to (0,\infty)$ be a non-decreasing unbounded function. An increasing function $T:(R_0,\infty)\to(0,\infty)$ is said to be of $\varphi$-order $\rho_\varphi(T)$ if
	$$
	\rho_\varphi(T)=\limsup_{r\to\infty}\frac{\log T(r)}{\log\varphi(r)}.
	$$
For a meromorphic function $f$ in $\mathbb{C}$, the $\varphi$-order of $f$ is defined as  
$\rho_\varphi(f)=\rho_\varphi(T(r,f))$, where $T(r,f)$ is the Nevanlinna characteristic function of $f$. 
Clearly, the classical order $\rho(f)$ and the logarithmic order $\rho_{\log}(f)$ follow as its special cases when choosing $\varphi(r)=r$ and $\varphi(r)=\log r$, respectively.
Thus the   $\varphi$-order is a general growth scale, and can be considered as a continuum between $\rho_{\log}(f)$ and $\rho(f)$
when we impose a global restriction
	\begin{equation}\label{general-restriction}
	\log r\leq \varphi(r)\leq r,\quad r\geq R_0.
	\end{equation}

For an entire function $f$, $T(r,f)$ can be replaced by $\log M(r,f)$ in the definitions of $\rho(f)$ and $\rho_{\log}(f)$ by the standard relation between $T(r,f)$ and $M(r,f)=\max_{|z|=r}|f(z)|$,  see \cite[p.~23]{Rubel}. The same is true for the $\varphi$-order, namely
	\begin{equation}\label{varphi-logM}
	\rho_\varphi(f)=\limsup_{r\to\infty}\frac{\log \log M(r,f)}{\log\varphi(r)},
	\end{equation}
provided that $\varphi(r)$ is subadditive, that is, $\varphi(a+b)\leq \varphi(a)+\varphi(b)$ for all $a,b\geq R_0$.
In particular, this gives  $\varphi(2r)\leq 2\varphi(r)$, which yields \eqref{varphi-logM}. We will see that many results involving the $\varphi$-order also require the subadditivity of $\varphi(r)$. 
\par Up to a normalization,
subadditivity is implied by concavity. Indeed, if $\varphi(r)$ is concave on $[0,\infty)$ and satisfies $\varphi(0)\geq 0$, then
for $t\in [0,1]$,
	$$
	\varphi(tx)=\varphi(tx+(1-t)\cdot 0)\geq t\varphi(x)+(1-t)\varphi(0)\geq t\varphi(x),
	$$
so that by choosing $t=\frac{a}{a+b}$ or $t=\frac{b}{a+b}$,
	\begin{equation}\label{subadditive}
	\begin{split}
	\varphi(a+b) &=\frac{a}{a+b}\varphi(a+b)+\frac{b}{a+b}\varphi(a+b)\\
	&\leq \varphi\left(\frac{a}{a+b}(a+b)\right)+\varphi\left(\frac{b}{a+b}(a+b)\right)\\
	&=\varphi(a)+\varphi(b),\quad a,b>0.
	\end{split}
	\end{equation}
As a non-decreasing, subadditive and unbounded function, $\varphi(r)$ satisfies
    $$
	\varphi(r)\leq \varphi(r+R_0)\leq \varphi(r)+\varphi(R_0)
	$$
for any $R_0\geq 0$. This yields $\varphi(r)\sim \varphi(r+R_0)$ as $r\to\infty$, and hence the normalization discussed above
can be assumed, whenever needed.

It should be noted that there exist differences between the order and the logarithmic order. For example, for a meromorphic function $f$ and for $a\in\mathbb{C}$, we have
	\begin{align}
	\rho_{\log}(N(r))&=\rho_{\log}(n(r))+1,\label{logorder-N-n}\\
	\rho(N(r))&=\rho(n(r)),\label{order-N-n}
 	\end{align}
where $n(r)=n(r,a,f)$ and $N(r)=N(r,a,f)$ denote the counting function and the integrated counting function of the zeros of $f-a$.
The identity \eqref{order-N-n}
is well-known, while \eqref{logorder-N-n} is proved in \cite[Theorem~4.1]{Chern}. The $\varphi$-order case in this situation along with some other
situations will be discussed later on.

The $\varphi$-order was defined for meromorphic functions in the unit disc, and it was used to measure the rate of the growth of solutions of linear differential equations (LDE's) in \cite{CHR}. However, LDE's in the complex plane seem not to be the right target of application for two reasons: (1) Transcendental solutions of LDE's with polynomial coefficients are known to be of positive order \cite{G-S-W}. (2) Most solutions of LDE's with transcendental coefficients are typically of infinite order. It is clear that the classical order is enough for these cases.

Instead of LDE's, the $\varphi$-order in the complex plane is applied to $q$-difference equations, because they have been considered earlier for functions of zero classical order \cite{BHKM} and for functions of finite logarithmic order \cite{CF,ZT}. Since the logarithmic order appears to be restrictive \cite[p.~267]{CF}, it would be of interest to see some results on the transition  between these two growth ranges.

To investigate $q$-difference equations, we recall $q$-difference analogue of the lemma on the logarithmic derivative \cite[Lemma~5.1]{BHKM}. It is a fundamental tool in $q$-difference operators and $q$-difference equations. Based on the lemma, we will obtain upper bounds for logarithmic $q$-differences in terms
of $\varphi(r)$.  

\begin{letterlemma}\label{BHMK.lemma}
Let $f$ be a meromorphic function such that $f(0)\neq 0, \infty$, and let $q\in\C\backslash\{0\}$. Then
	\begin{equation*}
	\begin{split}
	m\left(r,\frac{f(qz)}{f(z)}\right)\leq &\left(n(\lambda,f)+n\left(\lambda,\frac{1}{f}\right)\right)
	\left(\frac{|q-1|^{\delta}(|q|^{\delta}+1)}{\delta(1-\delta)|q|^{\delta}}+\frac{|q-1|r}{\lambda-|q|r}+\frac{|q-1|r}{\lambda-r}\right)\\
	&+\frac{4|q-1|r\lambda}{(\lambda-r)(\lambda-|q|r)}\left(T(\lambda,f)+\log^+\left|\frac{1}{f(0)}\right|\right),
	\end{split}
	\end{equation*}
where $z=re^{i\phi}$, $\lambda>\max\{r, |q|r\}$ and $0<\delta<1$.
\end{letterlemma}

Taking $\lambda=r^2$ in Lemma~\ref{BHMK.lemma}, \cite[Theorem~2.2]{ZT} gives a $q$-difference analogue of the lemma on the logarithmic derivative as
	\begin{equation}\label{ZT-estimate}
	m\left(r,\frac{f(qz)}{f(z)}\right)=O\left((\log r)^{\rho_{\log}(f)-1+\varepsilon}\right)
	\end{equation}
for each $\varepsilon>0$, if $\rho_{\log }(f)<\infty$. The estimate \eqref{ZT-estimate} is then used to prove \cite[Theorem 1.1]{ZT} on linear $q$-difference equations, which is stated as follows.

\begin{lettertheorem}\label{geq.theorem}
Let $a_0(z),\ldots,a_n(z)$ be entire functions of finite logarithmic order such that, for some $i\in\{0,\ldots,n\}$,
    $$
      \rho_{\log}(a_i)>\max_{j\neq i} \{\rho_{\log}(a_j)\},
    $$
 and let $q\in\C\backslash\{0\}$ be such that $|q|\neq 1$. If $f$ is a meromorphic solution of
	$$
	\sum_{j=0}^na_j(z)f(q^jz)=0,
	$$
then $\rho_{\log}(f)=\rho_{\log}(a_i)+1$.
\end{lettertheorem}

Among our main objectives is to find  $\varphi$-order analogues of the estimate \eqref{ZT-estimate} and of Theorem~\ref{geq.theorem}.
To this end, we globally assume that a non-decreasing function $s:(R_0,\infty)\to(0,\infty)$ satisfies
	\begin{equation}\label{assumption}
	r< s(r)\leq r^2,\quad r\geq R_0.
	\end{equation}
The function $s(r)$ takes the role of $\lambda$ in Lemma~\ref{BHMK.lemma}.
Suitable test functions for $\varphi(r)$ and $s(r)$ then are, for example,
	\begin{equation}\label{test-functions}
	\varphi(r)=\log^\alpha r,\quad \varphi(r)=\exp(\log^\beta r),\quad \varphi(r)=r^\beta,
	\end{equation}
along with $s(r)=r\log r$ and $s(r)=r^\alpha$, where $\alpha\in(1,2]$ and $\beta\in (0,1]$.

 Meromorphic functions of fast growth are already thoroughly investigated by the classical order. Thus, in this paper, we refer to a meromorphic function $f$ of finite $\varphi$-order as a slowly growing function $f$ in the sense that
    \begin{equation}\label{global-rho<1}
    \rho(f)<1.
    \end{equation}

This paper is organized as follows. In Section~\ref{results}, we state the results of $q$-difference equations for the $\varphi$-order.
We devote Section~\ref{varphi-calculus} to giving a $\,\,\,\,\,\,\,\,\,\,\,\,\,\,\,\,\,\,\,\,q$-difference analogue of Lemma~\ref{BHMK.lemma} in terms of the $\varphi$-order, which is one of the most important individual tools to prove Theorem~\ref{growth-coe} below. Section~\ref{varphi-calculus}  also contains a  
discussion on general growth parameters and growth properties related to the $\varphi$-order.
In Section~\ref{exp-conv-sec}, we introduce the $\varphi$-exponent of convergence and prove some properties related to it.
 Finally, the proofs of Theorems~\ref{growth-coe} and~\ref{rho f-rho a}  are given in Sections~\ref{proof of Th4.1} and~\ref{proof of Theorem 4.2}, respectively.

%%%%%%%%%%%%%%%%%%%%%%%%%%%%%%%%%%%%%%
% SECTION 2
%%%%%%%%%%%%%%%%%%%%%%%%%%%%%%%%%%%%%%
\section{Results on $q$-difference equations}\label{results}

We consider the growth of meromorphic solutions of  $q$-difference equations
	\begin{equation}\label{diffeqn}
	\sum_{j=0}^na_j(z)f(q^jz)=0
	\end{equation}
and of the corresponding non-homogeneous $q$-difference equations	
	\begin{equation}\label{diffeqn-non-homo}
	\sum_{j=0}^na_j(z)f(q^jz)=a_{n+1}(z),
	\end{equation}
where $a_0,\ldots,a_{n+1}$ are meromorphic functions. Moreover, we may assume  $a_0a_n\not\equiv 0$. The early results on these equations are reviewed in \cite{Chen}.

	Our results that follow depend on certain growth parameters defined as
		\begin{equation}\label{liminf}
	\alpha_{\varphi,s}=\liminf_{r\to\infty}\frac{\log \varphi(r)}{\log \varphi(s(r))},\quad
	\beta_\varphi=\limsup_{r\to\infty}\frac{\log\log r}{\log\varphi(r)}
	\end{equation}
and
	\begin{equation}\label{L}
	\gamma_{\varphi,s}=\liminf_{r\to\infty}\frac{\log\log \frac{s(r)}{r}}{\log \varphi(r)}.
	\end{equation}
Due to the assumptions \eqref{general-restriction} and \eqref{assumption},
we always have $ \alpha_{\varphi,s}\in[0,1]$, $\beta_\varphi\in[0,1]$ and $ \gamma_{\varphi,s}\in[-\infty,1]$.
From now on, we make a global assumption
    \begin{equation}\label{global-s/r}
    \liminf_{r\to\infty}\frac{s(r)}{r}>1,
    \end{equation}
which ensures that $ \gamma_{\varphi,s}\in[0,1]$. Further properties and relations among the growth parameters $\alpha_{\varphi,s},\beta_\varphi,\gamma_{\varphi,s}$ will be obtained in Section~\ref{general-parameters-sec}.

The first result on $q$-difference equations reduces to the logarithmic order case in \cite[Lemma~5.1]{ZT}
when  $s(r)=r^2$ and $\varphi(r)=\log r$.

\begin{theorem} \label{growth-coe}
Suppose that $a_0(z),\ldots,a_n(z)$ are meromorphic functions of finite
$\varphi$-order such that,  for some $i\in \{0,\ldots,n\}$,
	\begin{equation}\label{dominant-coeff}
	\rho_\varphi(a_i)>\max_{j\neq i}\{\rho_\varphi(a_j)\}.
	\end{equation}
Suppose that $\varphi(r)$ is subadditive. Let $\alpha_{\varphi,s}$ and $\gamma_{\varphi,s}$ be the constants in \eqref{liminf} and \eqref{L}, and let $q\in\C\setminus\{0\}$ and $\varepsilon>0$.
 \begin{itemize}
\item[\textnormal{(a)}]
Suppose that $\displaystyle\limsup_{r\to\infty}\frac{s(r)}{r}=\infty$ and that $s(r)$
is convex and differentiable. If $f$ is a meromorphic solution of \eqref{diffeqn}, then
	$$
	\rho_{\varphi}(f)\geq \max\left\{\rho_{\varphi}(f)-\alpha_{\varphi,s} \gamma_{\varphi,s},\, \alpha_{\varphi,s}\rho_{\varphi}(f)\right\}\geq \alpha_{\varphi,s} \rho_\varphi(a_i).
	$$
Moreover, if the coefficients $a_0(z),\ldots,a_n(z)$ are entire, then
    $$
    \rho_\varphi(f)\geq \alpha_{\varphi,s}\rho_\varphi(a_i)+\alpha_{\varphi,s}\gamma_{\varphi,s}.
    $$
 \item[\textnormal{(b)}]
 Suppose that $\displaystyle\limsup_{r\to\infty}\frac{s(r)}{r}<\infty$. If $f$ is a meromorphic solution of \eqref{diffeqn}, then
     $
 \rho_\varphi(f)\geq \rho_\varphi(a_i).
     $
\end{itemize}

\end{theorem}

The second result on $q$-difference equations reduces to the logarithmic order case in \cite[Lemma~5.3]{ZT} when choosing $s(r)=r^2$ and $\varphi(r)=\log r$, and to the classical order case in \cite[Theorem~3.5] {Heittokangas} when choosing $s(r)=2r$ and $\varphi(r)=r$. However, due to the global growth restriction \eqref{global-rho<1}, our result has only a partial overlap with \cite[Theorem~3.5] {Heittokangas}. This does not seem to be a big disadvantage since the growth of meromorphic functions $f$ of order $1\leq \rho(f)<\infty$ can be adequately expressed by using the classical order.

\begin{theorem}\label{rho f-rho a}
Let $a_0(z),\ldots,a_{n+1}(z)$ be meromorphic functions of finite 
\mbox{$\varphi$-order} and at least one of them is non-constant. Denote
    \begin{equation}\label{rho-varphi-coe}
	\rho_\varphi=\max_{0\leq j\leq n+1}\{\rho_\varphi(a_j)\}.
	\end{equation}
Suppose  $s(r)$ is convex and differentiable such that
    \begin{equation}\label{Young-condition}
    \limsup_{r\to\infty}\frac{r^2s'(r)}{s(r)^2}<\infty
    \end{equation}
and that $\varphi(r)$ is subadditive and differentiable and that one of the following holds:
    \begin{align}
 	\limsup_{r\to\infty}\frac{\varphi'(s(r))s'(r)r}{\varphi(s(r))}
 	&< \frac{1}{\lambda}, \label{assumption-limsup}\\
    \liminf_{r\to\infty}\frac{\varphi'(s(r))s'(r)r}{\varphi(s(r))}
  	&\geq \frac{1}{\lambda},   \label{assumption-limsup<liminf}
    \end{align}
     where $\lambda$ is the $\varphi$-exponent of convergence of zeros of $a_0$ to be defined in Section~\ref{exp-conv-sec}. Let $\alpha_{\varphi,s}>0$, $\beta_\varphi$  and $\gamma_{\varphi,s}$ be the constants in \eqref{liminf} and \eqref{L}, and let $q\in\C\setminus\{0\}$, $|q|\neq 1$.
If $f$ is a meromorphic solution of \eqref{diffeqn-non-homo}, then
     \begin{equation}\label{L-alpha-beta}
     \rho_\varphi(f)\leq \frac{\rho_\varphi}{\alpha_{\varphi,s}^2}
     -\frac{\gamma_{\varphi,s}}{\alpha_{\varphi,s}}+2\beta_\varphi.
     \end{equation}
  If all of $a_0(z),\ldots,a_{n+1}(z)$ are constants, then $\rho_\varphi(f)\leq \beta_\varphi$.
\end{theorem}

\begin{remark}
(a)\,Although there are many conditions in Theorem~\ref{rho f-rho a},  the result still makes sense since there are plenty of functions $\varphi(r)$ and $s(r)$ satisfying  \eqref{Young-condition}-\eqref{assumption-limsup<liminf}. Examples of such functions are the test functions $\varphi(r)$ in \eqref{test-functions} along with $s(r)=r\log r$ and $s(r)=r^\alpha$, where $\alpha\in(1,2]$ and $\beta\in (0,1]$.
\vskip 2mm

(b)\,If $\limsup_{r\to\infty}\frac{s(r)}{r}<\infty$, then \eqref{L-alpha-beta} reduces to
 $\rho_\varphi(f)\leq \rho_\varphi+2\beta_\varphi$ by Remark~\ref{gamma_s<0}(a) below. If $\gamma_{\varphi,s}=1$, then \eqref{L-alpha-beta} reduces to
$\rho_\varphi(f)\leq \rho_\varphi+1$ by Lemma~\ref{parameter-lemma} below.
\end{remark}

We have the following corollary from Remark~\ref{ilpo}(b) and Theorems~\ref{growth-coe}--\ref{rho f-rho a}.
\begin{corollary}
Let $a_0(z),\ldots,a_n(z)$ be entire functions of finite $\varphi$-order with \eqref{dominant-coeff}. Suppose that $\varphi(r)$ is subadditive and differentiable such that the limits
    \begin{equation}\label{3-limits}
    \lim_{r\to\infty}\frac{\log \varphi(r)}{\log \varphi(r^2)},
    \quad
    \lim_{r\to\infty}\frac{\log\log r}{\log\varphi(r)}
    \quad
    \text{and}\quad\lim _{r\to\infty}\frac{\varphi'(r)r}{\varphi(r)}
    \end{equation} 
    exist, and suppose that $s(r)$ is convex and differentiable satisfying \eqref{Young-condition}. Let $\beta_\varphi$  be the constant in \eqref{liminf}. If $f$ is a meromorphic solution of \eqref{diffeqn}, then
  $$
  \rho_\varphi(f)=\rho_\varphi(a_i)+\beta_\varphi.
  $$
\end{corollary}

  In the proof, one should choose $s(r)=r^2$ if $\beta_\varphi=0$, while $s(r)=2r$ if $\beta_\varphi>0$. Moreover, the existence of the limit $\lim _{r\to\infty}\frac{\varphi'(r)r}{\varphi(r)}$ implies the existence of the limit  $\lim _{r\to\infty}\frac{\varphi'(s(r))s'(r)r}{\varphi(s(r))}$. It is easy to check that the limits in \eqref{3-limits} exist for all test functions $\varphi(r)$ in \eqref{test-functions}.

%%%%%%%%%%%%%%%%%%%%%%%%%%%%%%%%%%%%%%%%%%%%%%
% SECTION 3
%%%%%%%%%%%%%%%%%%%%%%%%%%%%%%%%%%%%%%%%%%%%%%
\section{Growth results on the $\varphi$-order}\label{varphi-calculus}

We begin this section by stating and proving a $q$-difference analogue of the lemma on the logarithmic derivative for the $\varphi$-order,
which is one of our most important auxiliary results. We then proceed to discuss general growth parameters and
growth properties  related to the $\varphi$-order.

\par
We use the notation $r\geq R_0$ to express that the associated inequality is valid ''for all
$r$ large enough''. The notation $g(r)\lesssim h(r)$ means that there exists a constant $C\geq 1$
such that $g(r)\leq Ch(r)$ for all $r\geq R_0$. The notation $g(r)\asymp h(r)$ means that $g(r)\lesssim h(r)$ and $h(r)\lesssim g(r)$ hold simultaneously.

%%%%%%%%%%%%%%%%%%%%%%%%%%%%%%%%%%%%%%%%%%%%%%
% SUBSECTION 3.1
%%%%%%%%%%%%%%%%%%%%%%%%%%%%%%%%%%%%%%%%%%%%%%

\subsection{Lemma on the logarithmic $q$-difference}\label{log-difference-sec}

The following lemma  reduces to \cite[Theorem~2.2]{ZT} when choosing $s(r)=r^2$ and $\varphi(r)=\log r$.
Certain auxiliary functions such as $u(r)$ and $v(r)$ are constructed in the proof, which are not needed in proving
\cite[Theorem~2.2]{ZT}.
	
\begin{lemma}\label{q-difference-lemma}
Let $f$ be a meromorphic function of finite $\varphi$-order
$\rho_\varphi(f)$, and let $\varepsilon>0$ and $q\in\C\setminus\{0\}$.
\begin{itemize}
\item[\textnormal{(a)}] If $\displaystyle\limsup_{r\to\infty}\frac{s(r)}{r}=\infty$ and if $s(r)$
is convex and differentiable, then
	\begin{equation*}
	m\left(r,\frac{f(qz)}{f(z)}\right)
	=O\left(\frac{\varphi(s(r))^{\rho_\varphi(f)+\varepsilon}}{\log\frac{s(r)}{r}}\right).
    \end{equation*}
\item[\textnormal{(b)}] If $\displaystyle\limsup_{r\to\infty}\frac{s(r)}{r}<\infty$ and if
$\varphi(r)$ is subadditive, then
    \begin{equation*}
	m\left(r,\frac{f(qz)}{f(z)}\right)=O\left(\varphi(r)^{\rho_\varphi(f)+\varepsilon}\right).
    \end{equation*}
\end{itemize}
\end{lemma}

\begin{proof}
We first suppose that $f(0)\neq 0,\infty$.

(a) It follows from the assumptions on $s(r)$ that
	\begin{equation}\label{limit=infinity}
	\lim_{r\to\infty}\frac{s(r)}{r}=\infty.
	\end{equation}	
Indeed, suppose on the contrary to this claim that there exists a sequence $\{r_n\}$ of positive real
numbers tending to infinity and a constant $C\in (1,\infty)$ independent of $n$ such that
$s(r_n)\leq Cr_n$ for all
$n\in\N$. Choose $r\in [r_n,r_{n+1}]$, where $n\in\N$ is arbitrary but fixed. Then
$r=(1-t)r_n+tr_{n+1}$ for some $t\in [0,1]$. By the convexity of $s(r)$,
	$$
	s((1-t)r_n+tr_{n+1})\leq ts(r_{n+1})+(1-t)s(r_n)\leq C(tr_{n+1}+(1-t)r_n),
	$$
or in other words, $s(r)\leq Cr$ for all $r\in [r_n,r_{n+1}]$. Since $n\in\N$ is arbitrary and
$C$ is independent of $n$, we have $s(r)\leq Cr$ for all $r\geq r_1$, which is a contradiction.
This proves \eqref{limit=infinity}.

Next we show that $s(r)\leq 2s(r-1)$ for all $r\geq R_0$. Suppose on the contrary to this
claim that there exists an infinite sequence $\{t_n\}$ of positive real numbers tending to
infinity such that $s(t_n)>2s(t_n-1)$ for all $n\in\N$. By the mean value theorem and 
by convexity, for every $n$ large enough there exists a constant $c_n\in (t_n-1,t_n)$ such that
	\begin{equation}\label{squeeze}
	s(t_n-1)<s(t_n)-s(t_n-1)=s'(c_n)\leq s'(t_n).
	\end{equation}
By making use of \eqref{limit=infinity}, \eqref{squeeze}, L'H$\hat{\text{o}}$pital's rule
and the squeeze theorem for divergent sequences, we conclude that
    $$
    \lim_{n\to\infty}\frac{s(t_n)}{t_n^2}=\lim_{n\to \infty}\frac{s'(t_n)}{2t_n}
    =\lim_{n\to \infty}\frac{s'(t_n)}{t_n-1}\frac{t_n-1}{2t_n}
    =\frac12\lim_{n\to \infty}\frac{s'(t_n)}{t_n-1}=\infty,
    $$
which violates our assumption $s(r)\leq r^2$.

We proceed to prove that there exist non-decreasing functions 
$u,v:[1,\infty)\to (0,\infty)$ with the following properties:
\begin{itemize}
\item[(1)] $r<u(r)<s(r)$ and $r<v(r)<s(r)$ for all $r\geq R_0$,
\item[(2)] $u(r)/r\to\infty$ and $v(r)/r\to\infty$ as $r\to\infty$,
\item[(3)] $2^{-1}s(r)\leq v(u(r))\leq s(r)$ for all $r\geq R_0$,
\item[(4)] $2\log (u(r)/r)\leq \log (s(r)/r)\leq 2u(r)/r$ for all $r\geq R_0$.
\end{itemize}
For example, if $s(r)=r^2$ as in the proof of \cite[Theorem~2.2]{ZT}, then we may choose $u(r)=r^{4/3}$
and $v(r)=r^{3/2}$.
In the general case, denote $s(r)=rh(r)$, where $h(r)$ is an increasing and unbounded function,
and define the step functions
	\begin{equation*}
	w(r) = s(n),\quad u(r) = n\log h(n),\quad r\in [n,n+1).
	\end{equation*}
Then $w(r)$ and $u(r)$ are left-endpoint approximations for $s(r)$ and $r\log h(r)$, respectively, on the
intervals $[n,n+1)$. Define
	$$
	v(r)=\frac{h(n)}{\log h(n)}r,\quad r\in\left[n\log h(n),(n+1)\log h(n+1)\right).
	$$
 It is immediate from the definitions of the functions $u(r)$ and $v(r)$ that
	$$
	v(u(r))=\frac{h(n)}{\log h(n)}u(r)=nh(n)=w(r),\quad r\in [n,n+1).
	$$
Moreover, $u(r)$ and $v(r)$ are non-decreasing
and satisfy (2) as well as the first property in (1). Since $r<u(r)$ and since $v(r)$ is
increasing, we get $r<v(r)<v(u(r))=w(r)\leq s(r)$. This verifies the second
property in (1).

By the inequality $s(r)\leq 2s(r-1)$ and by the definition of the
function $w(r)$, we have $s(r)\geq w(r)\geq s(r-1)\geq 2^{-1}s(r)$ for all $r\geq R_0$. So (3) is confirmed. Since $\log^2x\leq x$ for all $x$ large enough, we obtain
	\begin{equation*}
	2\log (u(r)/r) \leq 2\log\log h(n)\leq \log h(n) \leq \log h(r)
	= \log (s(r)/r),
	\end{equation*}
where $r\in [n,n+1)$ and $n$ is large enough. This verifies the first inequality in (4). From $s(n+1)\leq 2s(n)$, we obtain $h(n+1)\leq \frac{2n}{n+1} h(n)\leq 2h(n)$. Thus
	\begin{align*}
	r\log h(r) &\leq (n+1)\log h(n+1)\leq (n+1)\log (2h(n))\\
	&\leq 2n\log h(n)=2u(r),\quad r\in [n,n+1),
	\end{align*}
where $n$ is large enough. This gives the second inequality in (4).

It is clear by (2) that
	$$
	\frac{r}{u(r)-|q|r}+\frac{r}{u(r)-r}
	=\frac{1}{u(r)/r-|q|}+\frac{1}{u(r)/r-1}\to 0,\quad r\to\infty,
	$$
 and that
	$$
	\frac{ru(r)}{(u(r)-|q|r)(u(r)-r)}\leq
 	\frac{2r}{u(r)},\quad r\geq R_0.
	$$
 Hence, choosing $\lambda=u(r)$ in Lemma~\ref{BHMK.lemma}, we obtain
	\begin{equation*}\label{m-n(u)}
	m\left(r,\frac{f(qz)}{f(z)}\right)\lesssim
	n(u(r),f)+n\left(u(r),\frac{1}{f}\right)+\frac{r}{u(r)}\cdot
	T(u(r),f).
	\end{equation*}
Since
	$$
	N(v(r),f)-N(r,f)=\int_r^{v(r)}\frac{n(t,f)}{t}\, dt
	\geq n(r,f)\log\frac{v(r)}{r},
	$$
we get
    $$
	n(r,f)\leq \frac{N(v(r),f)}{\log\frac{v(r)}{r}}
	\leq \frac{T(v(r),f)}{\log\frac{v(r)}{r}},
	$$
and similarly for $n(r,1/f)$. Thus, using properties (1), (3) and (4),
	\begin{equation*}
	\begin{split}
	m\left(r,\frac{f(qz)}{f(z)}\right)&\lesssim \frac{T(s(r),f)}{\log\frac{s(r)}{u(r)}}+
	\frac{T(u(r),f)}{u(r)/r}\\
	&\leq\frac{T(s(r),f)}{\log\frac{s(r)}{r}-\log \frac{u(r)}{r}}+\frac{2T(s(r),f)}{\log\frac{s(r)}{r}}
	\leq \frac{4T(s(r),f)}{\log\frac{s(r)}{r}}.
	\end{split}
	\end{equation*}
 The assertion now follows from the definition of the $\varphi$-order.

(b) By the assumptions there exists a $C\in (1,\infty)$ such that
$r<s(r)<Cr$ for all $r\geq R_0$. Choosing $\lambda=Br$ in Lemma~\ref{BHMK.lemma} for $B=2\max\{|q|,C\}$,  we infer
	\begin{equation}\label{m-esti-Br}
	m\left(r,\frac{f(qz)}{f(z)}\right)\lesssim
	n(Br,f)+n\left(Br,\frac{1}{f}\right)+T(Br,f).
	\end{equation}
Using the standard estimate
	\begin{equation}\label{N-2r geq n-r}
	N(2r,f)-N(r,f)=\int_{r}^{2r}\frac{n(t,f)}{t}\, dt
	\geq n(r,f)\log 2,
	\end{equation}
and similarly for the zeros of $f$, we obtain from \eqref{m-esti-Br} and
the definition of $\varphi$-order that
	\begin{equation}\label{m-esti-Br2}
	m\left(r,\frac{f(qz)}{f(z)}\right)\lesssim T(2Br,f)
	\lesssim \varphi(2Br)^{\rho_\varphi(f)+\varepsilon}.
	\end{equation}
Let $m=[2B]+1$, then the subadditivity of $\varphi(r)$ implies that
$\varphi(2Br)\leq \varphi(mr)\leq m\varphi(r)$, so the assertion
follows from \eqref{m-esti-Br2}.

If $f(0)=0$ or $\infty$, there exists $k\in\Z$ such that $g(z)=z^kf(z)$ satisfies $g(0)\neq 0,\infty$. Thus conclusions (a) and (b) hold for $g(z)$. The definition of $g$ yields $T(r,g)\lesssim T(r,f)$, thus  $\rho_\varphi(g)\leq\rho_\varphi(f)$. At the same time, we have
   \begin{equation*}
   m\left(r,\frac{f(qz)}{f(z)}\right)\asymp m\left(r,\frac{g(qz)}{g(z)}\right).
   \end{equation*}
Combining these facts completes the proof.
\end{proof}

\begin{remark}\label{re-1}
(a)\,The expression $\log\frac{s(r)}{r}$ in Lemma~\ref{q-difference-lemma}(a) tends to
infinity by \eqref{limit=infinity}, and hence its presence makes the estimate in Part~(a) stronger.
In Lemma~\ref{q-difference-lemma}(b), the situation is different, and we have two possibilities:
	$$
	\liminf_{r\to\infty}\frac{s(r)}{r}>1 \quad\textnormal{and}\quad
	\liminf_{r\to\infty}\frac{s(r)}{r}=1.
	$$
In the former case $r<C_1r\leq s(r)\leq C_2r$ holds for some $1<C_1<C_2<\infty$ and for all $r\geq R_0$.
Hence $0<\log C_1\leq \log \frac{s(r)}{r}\leq \log C_2<\infty$ for all $r\geq R_0$. The latter case is prevented by the global assumption \eqref{global-s/r}, but let us consider this possibility briefly.
If $s(r)$ is assumed to be convex, it follows that
	\begin{equation*}\label{limit=1}
	\lim_{r\to\infty}\frac{s(r)}{r}=1.
	\end{equation*}
Indeed, suppose on the contrary to this claim that 	there exists a sequence $\{r_n\}$ of positive real
numbers tending to infinity and a constant $C\in (1,\infty)$ independent of $n$ such that
$s(r_n)\geq Cr_n$ for all $n\in\N$. By the assumption $\liminf_{r\to\infty}\frac{s(r)}{r}=1$ and \eqref{assumption}, there exists another sequence $\{t_n\}$ of positive real numbers tending to infinity such that $t_n<s(t_n)<Ct_n$ for all $n$ large enough, say $n\geq N$.
Clearly there exist $m,k>N$ such that $t_N<r_m<t_k$, and we see that $s(r)$ is not convex
on the interval $[t_N,t_k]$, which is a contradiction. 	
	
	\vskip 2mm
(b)\,Suppose that a function $s:(R_0,\infty)\to(0,\infty)$ satisfies the inequality $s(r)\leq 2s(r-1)$
for all $r\geq R_0\geq 1$ as in the proof of Lemma~\ref{q-difference-lemma}. For example, the test
functions $s(r)=r^\alpha$, $\alpha\in(1,2]$ and $s(r)=r\log r$ satisfy this inequality. Let $r\geq R_0+1$.
Then there exists an $N\in\N$ such that $R_0+N\leq r<R_0+N+1$, and so we obtain
	$$
	s(r)\leq 2s(r-1)\leq 2^2s(r-2)\leq\cdots\leq 2^{N}s(R_0+1)=O(2^r).
	$$
This is the maximal growth rate for  $s(r)$ satisfying $s(r)\leq 2s(r-1)$.
Hence the inequality above does not impose a new growth restriction on $s(r)$ in addition to \eqref{assumption}.
Note that this inequality is not satisfied by $s(r)=e^r$.
\end{remark}

%%%%%%%%%%%%%%%%%%%%%%%%%%%%%%%%%%%%%%%%%%%%%%
% SUBSECTION 3.2
%%%%%%%%%%%%%%%%%%%%%%%%%%%%%%%%%%%%%%%%%%%%%%

\subsection{General growth parameters $\alpha_{\varphi,s}$, $\beta_\varphi$ and $\gamma_{\varphi,s}$}\label{general-parameters-sec}

Next, we discuss the possible values for $\alpha_{\varphi,s}$, $\beta$ and $\gamma_{\varphi,s}$ defined in \eqref{liminf} and \eqref{L}. The global assumption \eqref{global-s/r} is also discussed.

\begin{remark}\label{gamma_s<0}
(a)\,Suppose that $\varphi(r)$ is subadditive and that \mbox{$\limsup_{r\to\infty}\frac{s(r)}{r}<\infty$.} Then we see that there exists an integer $N\geq 2$ such that
$s(r)<Nr$ for all $r\geq R_0$. Hence, $\gamma_{\varphi,s}=0$, and the subadditivity of $\varphi(r)$ yields
    \begin{equation}\label{sub-varphi}
    \varphi(r)\leq\varphi(s(r))\leq\varphi(Nr)\leq N\varphi(r),\quad r\geq R_0.
    \end{equation}
Thus
    $$
    1\geq  \alpha_{\varphi,s}=\liminf_{r\to\infty}\frac{\log \varphi(r)}{\log \varphi(s(r))}\geq \liminf_{r\to\infty}\frac{\log \varphi(r)}{\log (N\varphi(r))}=1,
    $$
which implies $\alpha_{\varphi,s}=1$.

\vskip 2mm
(b)\,Without \eqref{global-s/r}, $\gamma_{\varphi,s}$ might not have a lower bound. For example, suppose that $s(r)=r+1/r^k$, where $k\geq 1$, then there exists $R_0>0$ such that $s(r)$ is increasing and convex when $r>R_0$. Take $\varphi(r)=\log r$,
then
    $$
    \gamma_{\varphi,s}=\liminf_{r\to\infty}\frac{\log\log\left(1+r^{-(k+1)}\right)}{\log\log r}=-\infty.
    $$
    
\vskip 2mm
(c)\,With \eqref{global-s/r}, it is possible for some $\varphi(r)$ that $\alpha_{\varphi,s}=0$. First start from the point $(r_1,\log r_1)$ on the curve $y=\log x$, next take the point $(s(r_1),s(r_1))$ on $y=x$.
Then we choose $r_2$ such that $\log r_2>s(r_1)$, and pick up two points $(r_2,\log r_2)$ and $(s(r_2),s(r_2))$. Following the step, we can get the sequence $\{r_n\}$, which increases to $\infty$ as
 $n\to\infty$, further $\log r_{n+1}>s(r_n)$ for $n\in\mathbb{N}$. Let $\varphi(r):(r_1,\infty)\to (\log r_1,\infty)$ be the increasing function whose graph $y=\varphi(x)$ is a polygonal path which consists of the line segment connecting points $(r_n,\log r_n), (s(r_n),s(r_n))$ and the line segment connecting points $(s(r_n),s(r_n)),(r_{n+1},\log r_{n+1})$ for every $n\in\mathbb{N}$. Since
    $$
    \lim_{n\to\infty}\frac{\log\varphi(r_n)}{\log\varphi(s(r_n))}=\lim_{n\to\infty}\frac{\log\log r_n}{\log s(r_n)}=0,
    $$
this implies that $\alpha_{\varphi,s}=0$.
\end{remark}

\begin{lemma}\label{parameter-lemma}
Let $\alpha_{\varphi,s}$, $\beta_\varphi$ and $\gamma_{\varphi,s}$ be the constants in \eqref{liminf} and \eqref{L}. If $\gamma_{\varphi,s}=1$, then $\alpha_{\varphi,s}=\beta_\varphi=1$.
\end{lemma}

\begin{proof}
By the definitions of $\gamma_{\varphi,s}$ and $\beta_\varphi$ and by $s(r)\leq r^2$, we have
    $$
    1=\gamma_{\varphi,s}\leq \liminf_{n\to\infty}\frac{\log\log r}{\log \varphi(r)}
	\leq \beta_\varphi\leq 1,
    $$
which gives us $\beta_\varphi=1$.
Hence, we obtain for any $\varepsilon>0$
	$$
	(1-\varepsilon)\log\varphi(r) \leq \log\log r\leq \log\varphi(r),\quad r\geq R_0,
	$$
provided that $\log r\leq \varphi(r)$.
This gives us
	$$
	1\geq \alpha_{\varphi,s}\geq \liminf_{r\to\infty}\frac{\log \varphi(r)}{\log\varphi(r^2)}
	\geq (1-\varepsilon)\liminf_{r\to\infty}\frac{\log\log r}{\log\log r^2}=1-\varepsilon,
	$$	
where we may let $\varepsilon\to 0^+$. Thus $\alpha_{\varphi,s}=1$.
\end{proof}

\begin{lemma}\label{zeta-alpha-geq-1/2}
Let $\alpha_{\varphi,s}$ be the constant in \eqref{liminf}. If the limit
    $
    \lim_{r\to\infty}\frac{\log \varphi(r)}{\log \varphi(r^2)}
    $
     exists, then we have $ \alpha_{\varphi,s}\in[1/2,1]$.
\end{lemma}

\begin{proof}
Set
     \begin{equation}\label{zeta}
     \zeta_\varphi=\lim_{r\to\infty}\frac{\log \varphi(r)}{\log \varphi(r^2)},
     \end{equation}
Clearly $0\leq\zeta_\varphi\leq\alpha_{\varphi,s}\leq 1$, which yields $\alpha_{\varphi,s}\geq 1/2$ once $\zeta_\varphi\geq 1/2$.
Assume on the contrary  that $\zeta_\varphi<1/2$. It follows from \eqref {zeta} that for any $\varepsilon>0$
	$$
       \log \varphi(r)\leq (\zeta_\varphi+\varepsilon)\log\varphi(r^2),\quad r \geq R_0.
	$$
 Then, for every $r\geq	R_0$, there exists an $N\in\N\cup\{0\}$ such that $r\in  [R_0^{2^N}, R_0^{2^{N+1}})$. Thus,
    \begin{equation}\label{log-var-geq-N-cons}
    \begin{split}
       \log\varphi(r)&\geq \frac{1}{\zeta_\varphi+\varepsilon}\log\varphi(r^{1/2})\geq \cdots\geq \left(\frac{1}{\zeta_\varphi+\varepsilon}\right)^N\log\varphi(r^{1/2^N})\\
    &\geq\left(\frac{1}{\zeta_\varphi+\varepsilon}\right)^N\log\varphi(R_0),
    \quad r\in  [R_0^{2^N}, R_0^{2^{N+1}}).
    \end{split}
    \end{equation}
Since $\zeta_\varphi<1/2$, then for $0<\varepsilon<1/2-\zeta_\varphi$, we have $2(\zeta_\varphi+\varepsilon)<1$. From \eqref{log-var-geq-N-cons}, we obtain
    $$
    \frac{\log\varphi(r)}{\log r}\geq \left(\frac{1}{2(\zeta_\varphi+\varepsilon)}\right)^N\frac{\log\varphi(R_0)}{\log R_0^2},\quad r\in  [R_0^{2^N}, R_0^{2^{N+1}}),
    $$
which implies
    $$
    \lim_{r\to\infty}\frac{\log\varphi(r)}{\log r}=\infty.
    $$
This contradicts with the assumption \eqref{general-restriction}. This completes the proof.
\end{proof}

\begin{lemma}\label{beta.lemma}
Let $\alpha_{\varphi,s}$, $\beta_\varphi$ be the constants in \eqref{liminf}. If $\beta_\varphi>0$ and the limit
    $\zeta_\varphi$ in \eqref{zeta}
     exists,  then $\alpha_{\varphi,s}=1$.
\end{lemma}

\begin{proof}
We claim that the constant $\zeta_\varphi$ in \eqref{zeta} satisfies $\zeta_\varphi=1$. Suppose on the contrary that $\zeta_\varphi<1$. Then for $0<\varepsilon<1-\zeta_\varphi$, we have $\zeta_\varphi+\varepsilon<1$. Following the proof of Lemma~\ref{zeta-alpha-geq-1/2}, for every $r\geq	R_0$, there exists an $N\in\N\cup\{0\}$ satisfying $r\in  [R_0^{2^N}, R_0^{2^{N+1}})$ and we obtain \eqref{log-var-geq-N-cons}.
 Since
    $$
     \log\log r \leq N\log 2+\log\log R_0^2,\quad r\in  [R_0^{2^N}, R_0^{2^{N+1}}),
    $$
    it follows that
    $$
    \frac{\log \log r}{\log \varphi(r)}\leq \frac{N\log 2+\log\log R_0^2}{\left(\frac{1}{\zeta_\varphi+\varepsilon}\right)^N\log\varphi(R_0)}\to 0,\quad r\to \infty.
    $$
This contradicts with $\beta_\varphi>0$, so the assertion follows from $\zeta_\varphi\leq\alpha_{\varphi,s}\leq 1$.
\end{proof}
\begin{remark}\label{ilpo}
(a)\,Assuming $\beta_\varphi=0$, we estimate the value of $\alpha_{\varphi,s}$ in two cases, where $r<s(r)\leq r^\eta$ for a constant $\eta\in(1,2]$.

 (a1)\,Suppose that $\varphi(r)$ is subadditive, and set
	$$
	\kappa=\limsup_{r\to\infty}\frac{\log r}{\log\varphi(r)}.
	$$
Clearly $\kappa\geq 1$. For every $r\geq R_0$, there exists an $N\in\N$ such that $\;\;\;\;\;\;\;\;\;\;\;\;\;\;N\leq r^{\eta-1}\leq N+1$.
Using subadditivity yields
	$$
	\varphi(s(r))\leq \varphi(r^{\eta-1}\cdot r)\leq \varphi((N+1)r)
	\leq (N+1)\varphi(r)\leq (r^{\eta-1}+1)\varphi(r),
	$$
which gives $\alpha_{\varphi,s}\geq \frac{1}{(\eta-1)\kappa+1}$. For example, if $\varphi(r)=r^\nu$, $0<\nu\leq 1$,
it is easy to see that $\beta_\varphi=0$ and $\alpha_{\varphi,s}\geq \frac{\nu}{\eta-1+\nu}$. Moreover, if $\eta<1+\nu$, then $\alpha_{\varphi,s}>1/2$. This shows that $\alpha_{\varphi,s}$ in Lemma~\ref{zeta-alpha-geq-1/2} can be strictly greater than $1/2$.

 (a2)\,In other way, we assume $\varphi(r^\eta)\leq \varphi(r)^\eta$ for all $r\geq R(\eta)$. This inequality is satisfied by all test functions $\varphi(r)$ in \eqref{test-functions}. We obtain
	\begin{equation}\label{alpha}
	1\geq \frac{\log\varphi(r)}{\log\varphi(s(r))}\geq \frac{\log\varphi(r)}{\log\varphi(r^\eta)}
	\geq \frac{1}{\eta}\geq\frac12.
	\end{equation}
In particular, if $s(r)=r\log r$, then the discussion above is valid for all $\eta>1$, and consequently
$\alpha_{\varphi,s}=1$ in \eqref{liminf}. Here we could be even less restrictive by assuming $\varphi(r^\eta)\leq \big(\varphi(r)\log\varphi(r)\big)^\eta$ for all $r\geq R_0$, in which case \eqref{alpha}
would hold asymptotically.

\vskip 2mm
(b)\,The parameters $\alpha_{\varphi,s}$ and $\gamma_{\varphi,s}$ depend not only on $\varphi(r)$ but also on $s(r)$. For some specific choices for $s(r)$, we obtain $\beta_\varphi=\gamma_{\varphi,s}$ and $\alpha_{\varphi,s}=1$. If $\beta_\varphi=0$ and $\varphi(r)$ is subadditive, we take $s(r)=2r$, then $\alpha_{\varphi,s}=1$ and $\beta_\varphi=\gamma_{\varphi,s}=0$. If $\beta_\varphi>0$ and if the limit
    $
    \zeta_\varphi
    $
      in \eqref{zeta} exists, then $\alpha_{\varphi,s}=1$ from Lemma~\ref{beta.lemma}.
In this case, 
    $$
    \gamma_{\varphi,s}=\liminf_{r\to \infty}\frac{\log\log r}{\log\varphi(r)}
    $$
    by choosing $s(r)=r^2$.
    Then $\beta_\varphi=\gamma_{\varphi,s}$ if and only if the limit
    \begin{equation}\label{lim-gamma=beta}
    \lim_{r\to\infty}\frac{\log\log r}{\log\varphi(r)}
    \end{equation}
    exists. Note that if $\varphi(r)$ is not smooth enough, then the limit \eqref{lim-gamma=beta} does not exist. For example,
    $\varphi(r)$ constructed in Remark~\ref{gamma_s<0}(c) satisfies 
	$$
   \lim_{n\to\infty}\frac{\log\log r_n}{\log \varphi(r_n)}=\lim_{n\to\infty}\frac{\log\log r_n}{\log\log r_n}=1,
    $$
    $$    \lim_{n\to\infty}\frac{\log\log s( r_n)}{\log \varphi(s(r_n))}=\lim_{n\to\infty}\frac{\log\log s(r_n)}{\log s(r_n)}=0.
    $$
    Thus, the limit superior and limit inferior of
	$\frac{\log\log r}{\log\varphi(r)}$ are $1$ and $0$ respectively.
    \end{remark}

%%%%%%%%%%%%%%%%%%%%%%%%%%%%%%%%%%%%%%%%%%%%%%
% SUBSECTION 3.3
%%%%%%%%%%%%%%%%%%%%%%%%%%%%%%%%%%%%%%%%%%%%%%

\subsection{General growth properties}\label{general-growth-sec}

The next result generalizes \cite[p.~14]{Boas} and \cite[Theorem~1.1]{Chern}.

\begin{lemma}\label{convergent-lemma}
Let $T:(0,\infty)\to(0,\infty)$ be an integrable, non-decreasing and unbounded function. Suppose that $\varphi(r)$ is differentiable.
Then $T$ is of $\varphi$-order $\rho\in [0,\infty)$ if and only if the integral
         \begin{equation}\label{integral}
         \int^\infty\frac{\varphi'(t)T(t)}{\varphi(t)^{\mu+1}}\, dt
         \end{equation}
is convergent for $\mu>\rho$ and divergent for $\mu<\rho$. Moreover, if $\rho_\varphi(T(r))=\infty$,
then the integral in \eqref{integral} is divergent for every $\mu\in\R$.
\end{lemma}

\begin{proof}
We prove the result in three steps.

(i) Suppose that $\rho_{\varphi}(T)=\rho<\infty$. For $\mu=\rho+2\varepsilon>\rho$, there exists an
$r_\varepsilon>e$ such that $T(r)<\varphi(r)^{\rho+\varepsilon}$ for $r>r_\varepsilon$. It follows that
       \begin{eqnarray*}
        \int^{\infty}\frac{\varphi'(t)T(t)}{\varphi(t)^{\mu+1}}\, dt
        &=&\int^{r_\varepsilon}\frac{\varphi'(t)T(t)}{\varphi(t)^{\mu+1}}\, dt
        +\int_{r_\varepsilon}^{\infty}\frac{\varphi'(t)T(t)}{\varphi(t)^{\mu+1}}\, dt\\
        &\leq&  O(1)+\int_{r_\varepsilon}^{\infty}\frac{\varphi'(t)}{\varphi(t)^{1+\varepsilon}}dt
        <\infty,
       \end{eqnarray*}
that is, the integral \eqref{integral} is convergent for $\mu>\rho$.
We proceed to prove that the integral \eqref{integral} is divergent for $\mu<\rho.$ Since
	$$
	\int^\infty\frac{\varphi'(t)T(t)}{\varphi(t)^{\mu+1}}\, dt\geq \int^\infty\frac{\varphi'(t)T(t)}{\varphi(t)}\, dt\geq \int^\infty\frac{\varphi'(t)}{\varphi(t)}\, dt=\infty,\quad \mu\leq 0,
	$$
we see that the integral \eqref{integral} is divergent for $\mu\leq 0.$
Suppose on the contrary to our claim that the integral \eqref{integral} is convergent for some $\mu\in (0,\rho).$ Clearly,
	$$
	\frac{1}{\varphi(r)^\mu}=-\int_r^\infty d\left(\varphi(t)^{-\mu}\right)
	=\int_r^\infty\frac{\mu\varphi'(t)}{\varphi(t)^{\mu+1}}\, dt.
	$$
By the assumptions on $T$, we obtain
	\begin{equation}\label{ratio-finite}
	\frac{T(r)}{\mu\varphi(r)^\mu}
	=T(r)\int_r^\infty\frac{\varphi'(t)}{\varphi(t)^{\mu+1}}\, dt
	\leq \int_r^\infty\frac{\varphi'(t)T(t)}{\varphi(t)^{\mu+1}}\, dt<\infty.
	\end{equation}
Hence, the $\varphi$-order $\rho$ of $T(r)$ is at most $\mu$, which contradicts with our assumption $\mu<\rho.$

(ii) Suppose that the integral \eqref{integral} is convergent for $\mu>\rho$ and divergent for $\mu<\rho$. It follows from \eqref{ratio-finite} that the $\varphi$-order of $T(r)$ satisfies $\;\;\;\;\;\;\;\;\;\;\rho_{\varphi}(T)\leq \mu=\rho+\varepsilon$ for any positive number $\varepsilon$, and hence $\rho_{\varphi}(T)\leq \rho$.

Assuming that $\rho_{\varphi}(T)< \rho$,  we can write $\rho_{\varphi}(T)=\rho-2\epsilon$ for some $\epsilon>0$. From (i), the integral \eqref{integral} is convergent for $\mu=\rho_{\varphi}(T)+\epsilon=\rho-\epsilon$, which contradicts the assumption that the integral \eqref{integral} is convergent for $\mu>\rho$. It follows that $\rho_{\varphi}(T)=\rho$.

(iii) Finally suppose that $\rho_\varphi(T)=\infty$. A simple modification of the reasoning
in (i) shows that \eqref{integral} diverges for every $\mu\in\R$.
\end{proof}

For a meromorphic function $f$ of finite classical order, it is well-known that $\rho(f')=\rho(f)$, whereas for a transcendental meromorphic function $f$ of finite logarithmic order, \cite[Theorem~6.1]{Chern} shows that
$\rho_{\log}(f')=\rho_{\log}(f)$. In order to preserve the $\varphi$-order in differentiation, we need to impose
some restrictions that are automatically true for the classical order and for the logarithmic order.

Before stating our result, we state the following variation of  Chuang's \mbox{inequality}, which follows easily from the proof of \cite[Theorem~2.6.1]{zheng}: \emph{ Let $f$ be a meromorphic function with $f(0)\neq\infty$. Then for $0<r<R$, we have
    \begin{equation}\label{chuang's ine}
    T(r,f)\lesssim \frac{R}{R-r} \left(1+\log\frac{R}{R-r}\right)T(R,f')+\log^+R+1.
    \end{equation}
}
\begin{lemma}
Let $f$ be a transcendental meromorphic function of finite \mbox{$\varphi$-order.}
Assume that one of the following assumptions holds:
\begin{itemize}
\item[\textnormal{(a)}]
 $\varphi(r)$ is subadditive.
\item[\textnormal{(b)}]
Suppose that the constant $\alpha_{\varphi,s}$ in \eqref{liminf} satisfies
$\alpha_{\varphi,s}= 1$,  and that $s(r)$ is a non-decreasing function such that
    \begin{equation}\label{assumption-s}
    \rho_{\varphi}\left(\frac{s(r)}{s(r)-r}\right) = 0.
	\end{equation}
	\end{itemize}	
Then $\rho_\varphi(f')=\rho_\varphi(f)$.
\end{lemma}

\begin{proof}
If $f(0)=\infty$, there exists a $k\in\N$ such that $g(z)=z^kf(z)
$ satisfies $g(0)\neq \infty$, and
$T(r,g)=(1+o(1))T(r,f)$. Hence we suppose that $f(0)\neq \infty.$

(a) Choosing $R=2r$ in \eqref{chuang's ine}, we have
	$
  	T(r,f)\lesssim T(2r,f')+\log^+ r.
	$
Since $f$ is a transcendental meromorphic function of finite $\varphi$-order, then
    $$
    T(r,f')\lesssim T(r,f)+m\left(r,\frac{f'}{f}\right)\lesssim T(r,f)+\log r\lesssim T(r,f).
    $$
Thus, the subadditivity of $\varphi$ and the two inequalities above yield the assertion.

(b) Choosing $R=s(r)$ in \eqref{chuang's ine}, we have
	\begin{equation*}
  	T(r,f)\lesssim \frac{s(r)}{s(r)-r}\left(1+\log \frac{s(r)}{s(r)-r}\right) T(s(r),f')+\log^+ r.
	\end{equation*}
Then $\alpha_{\varphi,s}=1$ and \eqref{assumption-s} give us
	  \begin{equation}\label{f-leq-f'}
	  \rho_\varphi(f)\leq 2\rho_{\varphi}\left(\frac{s(r)}{s(r)-r}\right)+\limsup_{r\to\infty}\frac{\log T(s(r),f')}{\log\varphi(s(r))}\cdot\frac{1}{\alpha_{\varphi,s}}=\rho_\varphi(f').
	 \end{equation}
An estimate by Gol'dberg-Grinshtein \cite[Corollary~3.2.3]{CY} is
	\begin{equation*}
	\begin{split}
  	T(r,f')
  	& \lesssim  T(r,f)+\log^+ T(s(r),f)+\log^+\frac{s(r)}{r(s(r)-r)}\\
  	&\lesssim T(s(r),f)+\log^+\frac{s(r)}{r(s(r)-r)}.
  	\end{split}
	\end{equation*}
Therefore, $\rho_\varphi(f')\leq\rho_\varphi(f)$.
Combining this with \eqref{f-leq-f'}, the assertion follows.
\end{proof}

%%%%%%%%%%%%%%%%%%%%%%%%%%%%%%%%%%%%%%%%%%%%%%
% SECTION 4
%%%%%%%%%%%%%%%%%%%%%%%%%%%%%%%%%%%%%%%%%%%%%%

\section{The $\varphi$-exponent of convergence}\label{exp-conv-sec}

Now we define the $\varphi$-exponent of convergence $\lambda_\varphi$ of a sequence $\{z_n\}$
of points with an unbounded modulus as follows: If for all $\mu>0$,
$$\sum_n \frac{1}{\varphi(|z_n|)^\mu}=\infty,$$
we set $\lambda_\varphi=\infty$; while if the infinite sum is finite for some $\mu>0$, we denote
	$$
	\lambda_\varphi=\inf\left\{\mu>0:\sum_n\frac{1}{\varphi(|z_n|)^\mu}<\infty\right\}.
	$$

%%%%%%%%%%%%%%%%%%%%%%%%%%%%%%%%%%%%%%%%%%%%%%
% SUBSECTION 4.1
%%%%%%%%%%%%%%%%%%%%%%%%%%%%%%%%%%%%%%%%%%%%%%

\subsection{Growth of the counting functions}

The following result generalizes \cite[p.~15]{Boas} and \cite[Theorem~3.1]{Chern}.

\begin{lemma}\label{n-lemma}
Suppose that $f$ is a non-constant meromorphic function and of finite
$\varphi$-order. Then, for each $a\in\widehat{\C}$, the $\varphi$-order
of $n(r,a,f)$ equals to the $\varphi$-exponent of convergence of the
$a$-points of $f$.
\end{lemma}

\begin{proof}
Let $\{z_n\}$ be the sequence of $a$-points of $f$ listed according to multiplicities and ordered by increasing modulus. Denote $n(r)=n(r,a,f)$ for short. By Riemann-Stieltjes integration and integration by parts,
	\begin{equation}\label{both}
	\sum_{r_0<|z_n|<R}\frac{1}{\varphi(|z_n|)^\mu} = \int_{r_0}^R\frac{dn(t)}{\varphi(t)^\mu}
	=\left[\frac{n(t)}{\varphi(t)^\mu}\right]^R_{r_0}+\mu\int_{r_0}^R\frac{\varphi'(t)n(t)}{\varphi(t)^{\mu+1}}\, dt.
	\end{equation}
Suppose that the sum on the left converges as $R\to\infty$. Since the first expression on the
right cannot diverge to $-\infty$ as $R\to\infty$, the non-negative integral on the right must converge as $R\to\infty$. Conversely, suppose
that the integral on the right converges as $R\to\infty$. Since
	\begin{equation}\label{varphi-infinity}
	\int_{r_0}^\infty \frac{\varphi'(t)}{\varphi(t)}\, dt=\infty,
	\end{equation}
we must have $\frac{n(t)}{\varphi(t)^\mu}\to 0$ as $t\to\infty$. Hence the sum on the left converges as $R\to\infty$. We have proved that, as $R\to\infty$, the sum on the left-hand side and the integral on the right-hand side of \eqref{both} converge
or diverge at the same time. The assertion now follows from Lemma~\ref{convergent-lemma} and the
definition of the $\varphi$-exponent of convergence.
\end{proof}

\begin{remark}\label{genus=0}
For a meromorphic function $f$ of finite $\varphi$-order, the global assumption \eqref{global-rho<1} yields $\lambda\leq \rho(f)<1$, where $\lambda$ is the classical exponent of convergence of the $a$-points of $f$. Therefore,
    \begin{equation}\label{n-r-usual-ex}
	n(r)=n(r,a,f)=o(r),\quad a\in\widehat{\C}.
    \end{equation}
 \end{remark}

We note that Lemma~\ref{n-lemma} could also be proved by modifying the proof of \cite[Theorem~3.1]{Chern}. The details are omitted.

We proceed to compare the $\varphi$-orders of
	$
	n(r)=n(r,a,f)
	$
	and
	\mbox{$
	 N(r)=N(r,a,f),
	$}
and aim to find a continuum between the logarithmic order and the classical order in \eqref{logorder-N-n} and \eqref{order-N-n}.
In the general case, the constants $\alpha_{\varphi,s}$, $\beta_\varphi$, $\gamma_{\varphi,s}$ in \eqref{liminf} and \eqref{L}
play a critical role. We give two results of which the first one is very
straight-forward and reduces to \eqref{logorder-N-n} when choosing $s(r)=r^2$ and $\varphi(r)=\log r$
and to \eqref{order-N-n} when choosing $s(r)=2r$ and $\varphi(r)=r$.

\begin{lemma}\label{N-lemma}
Suppose that $f$ is a non-constant meromorphic function and of finite
$\varphi$-order. Let $\alpha_{\varphi,s}$, $\beta_\varphi$ and $\gamma_{\varphi,s}$ be the constants in \eqref{liminf} and \eqref{L}. Then the $\varphi$-orders of $n(r)=n(r,a,f)$ and $N(r)=N(r,a,f)$ satisfy
	\begin{align}
	\rho_\varphi(N(r)) &\leq
	\rho_\varphi(n(r))+\beta_\varphi,\label{up}\\
	\rho_\varphi(N(r)) &\geq \alpha_{\varphi,s}\rho_\varphi(n(r))
	+\alpha_{\varphi,s} \gamma_{\varphi,s}.\label{down}
	\end{align}
\end{lemma}

\begin{proof}
Since
	\begin{equation}\label{N leq n log r}
	N(r)=\int_1^{r}\frac{n(t)}{t}\, dt+O(\log r)
	\leq n(r)\log r+O(\log r),
	\end{equation}
it follows that
	$$
	\frac{\log N(r)}{\log\varphi (r)}\leq \frac{\log n(r)}{\log\varphi (r)}
	+\frac{\log\log r}{\log\varphi (r)}+o(1).
	$$
This yields \eqref{up}. On the other hand,
   \begin{equation}\label{N geq n}
	N(s(r),f)\geq\int_r^{s(r)}\frac{n(t)}{t}\, dt
	\geq n(r)\log\frac{s(r)}{r},
	\end{equation}
so that
	$$
	\frac{\log N(s(r))}{\log \varphi (s(r))}\geq
	\frac{\log\varphi(r)}{\log\varphi(s(r))}\cdot\left(\frac{\log n(r)}{\log\varphi(r)}
	+\frac{\log\log\frac{s(r)}{r}}{\log\varphi(r)}\right).
	$$
This yields \eqref{down}.
\end{proof}

\begin{remark}\label{N=n+beta}
(a)\,If $\beta_\varphi=0$, $\varphi(r)$ is subadditive and we take $s(r)=2r$, then $\alpha_{\varphi,s}=1$ and $\beta_\varphi=\gamma_{\varphi,s}=0$ by Remark~\ref{ilpo}(b).
 If $\beta_\varphi>0$, the limits
    $
    \lim_{r\to\infty}\frac{\log \varphi(r)}{\log \varphi(r^2)}
    $
     and \eqref{lim-gamma=beta} exist, and we take $s(r)=r^2$, then $\alpha_{\varphi,s}=1$ and $\beta_\varphi=\gamma_{\varphi,s}$ by  Remark~\ref{ilpo}(b).
  Thus,
    $$
    \rho_{\varphi}(N(r))=\rho_{\varphi}(n(r))+\beta_\varphi
    $$
    follows from Lemma~\ref{N-lemma}  in both cases.
    
\vskip 2mm
    (b)\,Suppose that $\varphi(r)$ is subadditive and that $\limsup_{r\to\infty}\frac{s(r)}{r}<\infty$. By Remark~\ref{gamma_s<0}(a), we have $\alpha_{\varphi,s}=1$ and $\gamma_{\varphi,s}=0$. It follows from Lemma~\ref{N-lemma} that
    $$
    \rho_\varphi(n(r))\leq \rho_\varphi(N(r)) \leq  \rho_\varphi(n(r))+\beta_\varphi.
    $$
\end{remark}
We obtain the following result  directly from Lemmas~\ref{n-lemma} and~\ref{N-lemma}.
   \begin{corollary}\label{rho>alp.lamb-alp.gam}
   Suppose that $f$ is a non-constant meromorphic function and of finite $\varphi$-order. Let $\alpha_{\varphi,s}$, $\beta_\varphi$ and $\gamma_{\varphi,s}$ be the constants in \eqref{liminf} and \eqref{L}, and let  $\lambda_{\varphi}\geq 0$ be the $\varphi$-exponent of convergence of the
$a$-points of $f$, where $a$ is arbitrary. Then we have
   \begin{equation}\label{coro-rho}
   \rho_\varphi(f)\geq \alpha_{\varphi,s}\lambda_\varphi+\alpha_{\varphi,s}\gamma_{\varphi,s}.
   \end{equation}
\end{corollary}
 The equality in \eqref{coro-rho} may hold even if the right-hand side is strictly positive. For example, if $\varphi(r)=\log r$ and $s(r)=r^2$, then $\alpha_{\varphi,s}=\gamma_{\varphi,s}=1$. Moreover, there exists a transcendental meromorphic function $f$ with $\rho_\varphi(f)=\lambda_\varphi+1=1$ by \cite[Theorem~ 7.3]{Chern}.

The second result on the $\varphi$-orders of $n(r)=n(r,a,f)$ and $N(r)=N(r,a,f)$  does not rely on $s(r)$ but only on $\varphi(r)$. It reduces to \eqref{logorder-N-n} when $\varphi(r)=\log r$ because the function $\psi_{\mu}(t)$ in \eqref{psi} below is then of the form $\psi_{\mu}(t)=\frac{\mu+1}{t\log t}$.

\begin{lemma}\label{N-lemma-0}
Suppose that $f$ is a non-constant meromorphic function such that $\rho_\varphi(n(r))\in [0,\infty]$ and that $\varphi(r)$ is concave and twice differentiable. Let $\tau\in [0,1]$, and suppose further that
for $\mu\in\R$, there
there exist constants $C_1(\mu)>0$ and $C_2(\mu)>0$ such that the function
	\begin{equation}\label{psi}
	\psi_\mu(t)=(\mu+1)\frac{\varphi'(t)}{\varphi(t)}-\frac{1}{t}-\frac{\varphi''(t)}{\varphi'(t)}
	\end{equation}
satisfies
	\begin{equation}\label{varphi2}
	0<C_1(\mu)\varphi(t)^{-\tau}\leq \psi_\mu(t)t\leq C_2(\mu)\varphi(t)^{-\tau}<\infty,\quad t\geq 0.
	\end{equation}
Then $\rho_\varphi(N(r))= \rho_\varphi(n(r))+\tau.$
\end{lemma}

\begin{proof}
Without loss of generality,
we may suppose that $\varphi(r)$ is defined on $[0,\infty)$ and satisfies $\varphi(0)=0$, see the discussion below the 
formula \eqref{subadditive}.

Set $\rho_{\varphi}(n(r))=\sigma$ and let $\mu\in\R$. Integrating by parts,
	\begin{equation}\label{second-integration}
	\begin{split}
	\int_{r_0}^R\frac{\varphi'(t)n(t)}{\varphi(t)^{\mu+1}}\, dt
	&=\int_{r_0}^R\frac{t\varphi'(t)}{\varphi(t)^{\mu+1}}\frac{n(t)}{t}\, dt\\
	&=\left[\frac{t\varphi'(t)N(t)}{\varphi(t)^{\mu+1}}\right]^R_{r_0}
	+\int_{r_0}^R\psi_\mu(t)t
	\frac{\varphi'(t)N(t)}{\varphi(t)^{\mu+1}}\, dt,
	\end{split}	
	\end{equation}
where $\psi_\mu(t)$ is given in \eqref{psi}, and it follows that
\begin{equation}\label{mu}
\int_{r_0}^R\psi_\mu(t)\,dt=\log\frac{\varphi(R)^{\mu+1}}{R\varphi'(R)}+C(\mu,r_0),
\end{equation}
where $C(\mu,r_0)$ is a constant.
 From \eqref{varphi2}, we find that the integral on the
right-hand side of \eqref{second-integration} converges as $R\to\infty$ if and only if
	\begin{equation}\label{alpha0}
	\int_{r_0}^\infty\frac{\varphi'(t)N(t)}{\varphi(t)^{\mu+\tau+1}}\, dt<\infty.
	\end{equation}
Let $\mu>\sigma$. Then the integral on the left-hand side of \eqref{second-integration} converges as $R\to\infty$ by Lemma~\ref{convergent-lemma}. Since the first expression on the right-hand side of \eqref{second-integration} cannot diverge to $-\infty$ as $R\to\infty$, the non-negative integral on the right-hand side must converge as $R\to\infty$, and thus \eqref{alpha0} holds for $\mu>\sigma$.

(1) Suppose first that $0<\sigma<\infty$ and let $0<\mu<\sigma$. Then the integral on the left-hand side of \eqref{second-integration}
diverges as $R\to\infty$ by Lemma~\ref{convergent-lemma}. Suppose that \eqref{alpha0} holds, and aim for a contradiction.
As $\varphi(r)$ is differentiable and concave, its derivative is non-increasing,
and so, using $\varphi(0)=0$,
	\begin{equation*}\label{derivative}
	\varphi(r)=\int_0^r\varphi'(t)\, dt\geq r\varphi'(r).
	\end{equation*}
Since $\mu>0$, inserting the formula above into \eqref{mu} leads to
	\begin{equation*}
	\begin{split}
	\int_{r_0}^R\psi_\mu(t)\, dt\geq \mu\log\varphi(R)+	C(\mu,r_0) \to\infty,\quad R\to\infty.
	\end{split}	
	\end{equation*}
Combining this with \eqref{alpha0}, we conclude that
$\frac{R\varphi'(R)N(R)}{\varphi(R)^{\mu+1}}\to 0$ as $R\to\infty$. Then the ~right-hand side of \eqref{second-integration} remains bounded as $R\to\infty$, which is a contradiction.
Thus, the integral in \eqref{alpha0} converges
for $\mu>\sigma$ and diverges for $\mu<\sigma$. The assertion $\rho_\varphi(N(r))= \sigma+\tau$
now follows from Lemma~\ref{convergent-lemma}.

(2) Suppose then that $\sigma=0$. Our task is to prove that $\rho_\varphi(N(r))=\tau$. Since  \eqref{alpha0} holds for $\mu>\sigma=0$, the assertion follows from Lemma~\ref{convergent-lemma} if  the integral in \eqref{alpha0} diverges for $\mu<0$.

First, suppose that $\frac{\varphi(R)^{\mu+1}}{R\varphi'(R)}$ is bounded for $-1\leq \mu<0$ as $R\to\infty$, and aim for a contradiction.
  Divide both sides of \eqref{varphi2} by $t>0$ and integrate from 
$r_0$ to $r$, then
	\begin{equation}\label{remark-asymptotic}
	\log\frac{\varphi(r)^{\mu+1}}{r\varphi'(r)}\asymp \int_{r_0}^r \frac{dt}{t\varphi(t)^\tau},
	\quad r\to\infty,
	\end{equation}
where the comparison constants may depend on $\mu$. Since the right-hand side of \eqref{remark-asymptotic} is positive, then 
    \begin{equation}\label{1<<C}
    1<\frac{\varphi(r)^{\mu+1}}{r\varphi'(r)}\leq C, \quad r\geq R_0,
    \end{equation}
  holds for some $C>1$. Thus, both sides of \eqref{remark-asymptotic} are bounded for $r\geq R_0$, which gives us $\tau>0$.
 Indeed, if $\tau=0$, then the right-side hand of \eqref{remark-asymptotic} is unbounded as $r\to\infty$, which is a contradiction.  We consider the following two subcases separately.
 \begin{itemize}
\item[\textnormal{(i)}]  If $-1<\mu<0$, it follows from  \eqref{general-restriction}  that
    $
   \varphi(r)^{\mu+1}\leq r^{\mu+1}$ for $r\geq R_0$,
     thus $\varphi(r)^{\mu+1}=o(r)$ as $r\to\infty$. Then for $r\geq R_0$, 
    $
    \varphi(r)<\varepsilon  r^{\frac{1}{^{\mu+1}}}
    $
    holds, which leads to
    $$
    \int_{R_0}^\infty \frac{dt}{t\varphi(t)^\tau}\geq\frac{1}{\varepsilon^\tau}\int_{R_0}^\infty t^{-\frac{\tau}{\mu+1}-1}dt=\frac{1}{\varepsilon^\tau}C(\mu,R_0,\tau),
	\quad r\geq R_0,
    $$
where $C(\mu,R_0,\tau)>0$ is a constant. Letting $\varepsilon\to 0^+$,  we get a contradiction.

 \item[\textnormal{(ii)}] If $\mu=-1$, then \eqref{1<<C} yields $\varphi'(r)<1/r$ for $r\geq R_0$. Hence, 
    \begin{equation*}
     \begin{split}
   \varphi(r)&=\int_{R_0}^r\varphi'(t)dt+\varphi(R_0) \leq\int_{R_0}^r\frac{dt}{t}+\varphi(R_0)
  \leq 2\log r
    \end{split}
    \end{equation*}
for all $r\geq R_0$. It follows from $0<\tau\leq 1$ that 
    $$
    \int_{R_0}^r \frac{dt}{t\varphi(t)^\tau}\geq  \int_{R_0}^r \frac{dt}{t(2\log t)^\tau}  \to\infty,\quad r\to\infty,
    $$
which is a contradiction.
 \end{itemize}
  Thus $\frac{\varphi(R)^{\mu+1}}{R\varphi'(R)}$ is not bounded for $-1\leq \mu<0$ as $R\to\infty$, and hence the integral in \eqref{mu} diverges  as $R\to\infty$. We may proceed as in Case (1) to conclude that 
the integral in \eqref{alpha0} diverges for $-1\leq \mu<0$.

Next, we proceed to prove that the integral in \eqref{alpha0} diverges for $\mu<-1$. Suppose on the contrary that \eqref{alpha0} holds, then 
    \begin{equation*}
    \begin{split}
    \infty>\int_{r_0}^\infty\frac{\varphi'(t)N(t)}{\varphi(t)^{\mu+\tau+1}}\, dt\geq N(r_0)\int_{r_0}^\infty \varphi(t)^{-\mu-\tau-1}\varphi'(t)\, dt,
    \end{split}
    \end{equation*}    
    which gives us $\tau+\mu>0$. We note that $\mu<-1$, then $\tau>-\mu>1$, which violates the assumption $\tau\leq 1$. Therefore, the integral in \eqref{alpha0} diverges for $\mu<0$.   
 Then the assertion $\rho_\varphi(N(r))=\tau$ follows.

(3) Finally suppose that $\sigma=\infty$. By \eqref{N-2r geq n-r} and the subadditivity of $\varphi(r)$, guaranteed by \eqref{subadditive}, we have in general $\rho_\varphi(n(r))\leq \rho_\varphi(N(r))$, see the proof of Lemma~\ref{q-difference-lemma}(b).
Since $\sigma=\infty$, the assertion follows.
\end{proof}

%%%%%%%%%%%%%%%%%%%%%%%%%%%%%%%%%%%%%%%%%%%%%%
% SUBSECTION 4.2
%%%%%%%%%%%%%%%%%%%%%%%%%%%%%%%%%%%%%%%%%%%%%%

\subsection{Minimal growth of canonical products}

Choosing $s(r)=2r$ and $\varphi(r)=\log r$, we find that \eqref{assumption-limsup} holds
for every $\lambda\in [0,\infty)$. Thus Lemma~\ref{canonical} below reduces to \cite[Lemma~3.3]{ZT} in
this case. Concerning the assumption \eqref{assumption-limsup<liminf}, we find by
\cite[Theorem~II]{Taylor} that
	\begin{equation*}\label{hospital}
	\liminf_{r\to\infty}\frac{\varphi'(s(r))s'(r)r}{\varphi(s(r))}
	\leq \limsup_{r\to\infty}\frac{\log\varphi(s(r))}{\log r}\leq 2.
	\end{equation*}
Here we have used the facts that $\varphi(r)\leq r$ and $s(r)\leq r^2$. Hence the inequality
in \eqref{assumption-limsup<liminf} does not occur if $\lambda<1/2$ but may occur if $\lambda\geq 1/2$.
The latter claim is easy to see by choosing $s(r)=2r$ and $\varphi(r)=r^\beta$ for $\beta\in [1/2,1]$.

\begin{lemma}\label{canonical}
Let $P(z)$ be a canonical product formed with the sequence $\{z_n\}$
that has a finite $\varphi$-exponent of convergence $\lambda\geq 0$, and let $\varphi(r)$
and $s(r)$ be differentiable such that \eqref{Young-condition} holds
and that one of \eqref{assumption-limsup} or \eqref{assumption-limsup<liminf} holds.
If $z$ lies outside of the discs $D_n=\left\{z:|z-z_n|\leq \frac{1}{r_n^{\lambda+\varepsilon}}\right\}$, where $r_n=|z_n|$, then
	$$
	\log\frac{1}{|P(z)|}=O\left(\varphi(s(r))^{\lambda+\varepsilon}\log r\right),\quad |z|=r.
	$$
\end{lemma}

\begin{proof}
This proof is a modification of the proof of~\cite[Lemma~3.3]{ZT},
which originates from \cite[Theorem~V.19]{Tsuji}. Keeping \eqref{global-rho<1} in mind, we may write
	\begin{align}\label{PP}
	\log\frac{1}{|P(z)|}&=\sum_{n=1}^\infty
	\log\frac{1}{\left|1-\frac{z}{z_n}\right|}
	=\sum_{s(r)\leq r_n}
	\log\frac{1}{\left|1-\frac{z}{z_n}\right|}+
	\sum_{s(r)>r_n}
	\log\frac{1}{\left|1-\frac{z}{z_n}\right|}\nonumber\\
	&=\sum_1+\sum_2.
	\end{align}
If $z$ lies outside of the discs $D_n$, then
	$$
	\log\left|\frac{z_n}{z_n-z}\right|\lesssim \log r_n.
	$$
Together with Lemma~\ref{n-lemma} and $ s(r)\leq r^2$, we obtain
	\begin{equation}\label{sum2}
	\sum_2\lesssim n(s(r))\log s(r)\lesssim \varphi(s(r))^{\lambda+\varepsilon}\log r.
	\end{equation}

Next we proceed to estimate $\sum_1$. It follows from \eqref{n-r-usual-ex} that
	$$
	0\leq n(t)\log\frac{1}{1-r/t}\leq \frac{rn(t)}{t-r}\overset{t\to\infty}{\longrightarrow}0.
	$$
Hence, a standard reasoning yields
	\begin{equation}\label{sum1}
	\begin{split}
	\sum_1 &\leq  \sum_{s(r)\leq r_n}\log\frac{1}{1-r/r_n}=\int_{s(r)}^\infty\log\frac{1}{1-r/t}\, dn(t)\\
	&\leq \int_{s(r)}^\infty \frac{n(t)}{t}\frac{r}{t-r}\, dt
	=O\left(r\int_{s(r)}^\infty\frac{n(t)}{t^2}\, dt\right).
	\end{split}
	\end{equation}

If \eqref{assumption-limsup} is fulfilled, then \cite[Theorem~II]{Taylor} and $r< s(r)$ yield
    \begin{equation*}\label{ass-limsup-1}
    \limsup_{r\to\infty}\frac{\log\varphi(r)}{\log r}\leq\limsup_{r\to\infty}\frac{\log\varphi(s(r))}{\log r}\leq\limsup_{r\to\infty}\frac{\varphi'(s(r))s'(r)r}{\varphi(s(r))}<\frac{1}{\lambda}.
    \end{equation*}
Hence, there exist $\varepsilon>0$, $R>0$ and $k\in(0,1)$ such that for $r>R$,
    \begin{equation}\label{log var(s)/log r}
    \max\left\{\frac{\log\varphi(r)}{\log r},
    \frac{\log\varphi(s(r))}{\log r},
    \frac{\varphi'(s(r))s'(r)r}{\varphi(s(r))}\right\}\leq\frac{k}{\lambda+\varepsilon}.
        \end{equation}
This implies $\int_{s(r)}^\infty\frac{\varphi(t)^{\lambda+\varepsilon}}{t^2}\, dt\to 0$  and $\varphi(s(r))^{\lambda+\varepsilon}=o(r)$,  as $r\to\infty$. Then it follows from \cite[Theorem~II]{Taylor}, \eqref{Young-condition} and  \eqref{log var(s)/log r} that
     \begin{equation}\label{0}
     \begin{split}
   \limsup_{r\to\infty}\frac{\int_{s(r)}^\infty\frac{\varphi(t)^{\lambda+\varepsilon}}{t^2}\, dt}{\frac{\varphi(s(r))^{\lambda+\varepsilon}}{r}}
    &\leq\limsup_{r\to\infty}\frac{-\frac{\varphi(s(r))^{\lambda+\varepsilon}}{s(r)^2}\cdot s'(r)}{\frac{(\lambda+\varepsilon)\varphi(s(r))^{\lambda+\varepsilon-1}\varphi'(s(r))s'(r)r-\varphi(s(r))^{\lambda+\varepsilon}}{r^2}}\\
    &=\limsup_{r\to\infty} \frac{\frac{r^2 s'(r)}{s(r)^2}}{1-\frac{(\lambda+\varepsilon)\varphi'(s(r))s'(r)r}{\varphi(s(r))}}<\infty.
    \end{split}
    \end{equation}
Lemma~\ref{n-lemma} yields $n(t)\leq \varphi(t)^{\lambda+\varepsilon}$, and then from \eqref{sum1} and \eqref{0}, we obtain  $\sum_1=O(\varphi(s(r))^{\lambda+\varepsilon})$. Together with \eqref{PP} and \eqref{sum2}, the assertion follows by assuming \eqref{assumption-limsup}.

If \eqref{assumption-limsup<liminf} is fulfilled, then \cite[Theorem~II]{Taylor} yields
    $
    \liminf_{r\to\infty}\frac{\log\varphi(s(r))}{\log r}\geq\frac{1}{\lambda}.
    $
Thus,
     \begin{equation*}\label{k-4}
     \frac{\log\varphi(s(r))}{\log r}\geq \frac{1}{\lambda+\varepsilon},\quad r>R,
     \end{equation*}
for any $\varepsilon>0$, which implies $r=O(\varphi(s(r))^{\lambda+\varepsilon})$, $r>R$. The global assumption \eqref{n-r-usual-ex} shows that the integral in the upper bound of \eqref{sum1} converges, and thus the conclusion follows from \eqref{PP}, \eqref{sum2} and \eqref{sum1}.
\end{proof}

\begin{remark}\label{radii-of-discs}
(a)\,The radii of the discs $D_n$ in Lemma~\ref{canonical} have a finite sum by the definition of the $\varphi$-exponent of convergence
and the fact that $\varphi(r_n)\leq r_n$.
Hence the collection of discs $D_n$ is an $R$-set by the definition \cite[p.~84]{Laine}.

\vskip 2mm
(b)\,Suppose that $s:[0,\infty)\to [0,\infty)$ is an increasing, differentiable and convex function satisfying
$s(0)=0$ and \eqref{assumption}. It is well-known that such a function $s(r)$ has a representation
	$$
	s(r)=\int_0^r \phi(t)\, dt,
	$$
where $\phi(t)$ is a non-decreasing function. If $\phi(t)$ is unbounded, then $s(r)$ is known as Young's function
\cite[Appendix~A-1]{Neveu}. Following the proof of \cite[Lemma~A-1-1]{Neveu}, we see that the conditions
	\begin{eqnarray}
	&&\limsup_{r\to\infty}\frac{rs'(r)}{s(r)} =\limsup_{r\to\infty}\frac{r\phi(r)}{s(r)}< \infty,\label{1e}\\
	&&\limsup_{r\to\infty}\frac{s(ar)}{s(r)} < \infty\ \textnormal{for some}\ a>1,\label{2e}\\
	&&\limsup_{r\to\infty}\frac{\phi(ar)}{\phi(r)} < \infty\ \textnormal{for some}\ a>1,\label{3e}	
	\end{eqnarray}
are equivalent. Indeed, one just has to verify that the inequalities in the proof of \cite[Lemma~A-1-1]{Neveu}
are valid for all $r$ large enough as opposed to all $r>0$.

This observation gives a new description of the assumption \eqref{Young-condition} in Lemma~\ref{canonical} in the sense that
if one of \eqref{1e},  \eqref{2e} or \eqref{3e} holds, then \eqref{Young-condition} holds by means of \eqref{assumption} and \eqref{1e}.

\vskip 2mm
(c)\,Suppose that $\varphi(r)$ is subadditive and that $\limsup_{r\to\infty}\frac{s(r)}{r}<\infty$. By \eqref{sub-varphi}, the conclusion in Lemma~\ref{canonical} will be replaced by
	$$
	\log\frac{1}{|P(z)|}=O\left(\varphi(r)^{\lambda+\varepsilon}\log r\right),\quad |z|=r.
	$$
\end{remark}

It is well-known from the second fundamental theorem that, for a transcendental meromorphic function $f$, $T(r,f)$ is typically dominated by three integrated counting functions. The next result shows that when $f$ is of finite $\varphi$-order, $T(r,f)$ can be dominated by two integrated counting functions. It reduces to \cite[Theorem~7.1]{Chern} when choosing $\varphi(r)=\log r$.

\begin{lemma}\label{aux-lemma}
Suppose that $f$ is a transcendental meromorphic function of finite $\varphi$-order $\rho_\varphi(f)$. Let $\lambda_1$ and $\lambda_2$ be the finite $\varphi$-exponent of convergence of  the zeros and the poles of $f$, respectively. Let $\varphi(r)$ be differentiable such that the following two cases occur simultaneously:
\begin{itemize}
\item[\textnormal{(a)}]
    $
    \limsup_{r\to\infty}\frac{\varphi'(r)r}{\varphi(r)}<\frac{1}{\lambda_1}
    $
 or
    $
    \liminf_{r\to\infty}\frac{\varphi'(r)r}{\varphi(r)}\geq\frac{1}{\lambda_1}
    $ 
    holds,  
\item[\textnormal{(b)}]
     $
    \limsup_{r\to\infty}\frac{\varphi'(r)r}{\varphi(r)}<\frac{1}{\lambda_2}
    $
or 
    $
    \liminf_{r\to\infty}\frac{\varphi'(r)r}{\varphi(r)}\geq\frac{1}{\lambda_2}
    $
    holds.
\end{itemize}
Suppose further that $\varphi(r)$ is subadditive such that the limits
    $
    \lim_{r\to\infty}\frac{\log \varphi(r)}{\log \varphi(r^2)}
    $
     and \eqref{lim-gamma=beta} exist. Let $\beta_\varphi$ be the constant in \eqref{liminf} and $\varepsilon>0$. Then for any two distinct extended complex values $a$ and $b$,
    \begin{equation}\label{4.6-1}
    T(r,f)\leq N\left(r,a,f\right)+N\left(r,b,f\right)+O\left(\varphi(r)^{\rho_\varphi(f)-\beta_\varphi+\varepsilon}\right)+O(\log r).
    \end{equation}
    Furthermore, if $\rho_\varphi(f)-\mu_\varphi(f)<\beta_\varphi$ and $\beta_\varphi>0$, then 
    $$
    T(r,f)\leq N\left(r,a,f\right)+N\left(r,b,f\right)+o(T(r,f)),
    $$
    where $\mu_\varphi(f)=\liminf_{r\to\infty}\frac{\log T(r,f)}{\log\varphi(r)}$.
    \end{lemma}

\begin{proof}
Let $P_1(z)$ and $P_2(z)$ be the canonical products formed from the zeros and the poles of $f$, respectively. It follows from the Hadamard's factorization theorem~\cite[p.~22]{Boas} and the global assumption \eqref{global-rho<1} that $f$ can be written as
    $$
    f(z)=Cz^m\frac{P_1(z)}{P_2(z)},
    $$
where $C>0$, $m\in\N\cup\{0\}$. Using the Borel-Valiron inequality \cite[p.~56]{GO} and Lemma~\ref{n-lemma}, we may proceed similarly as in Lemma~\ref{canonical},
     \begin{equation}\label{SMT}
    \begin{split}
    T(r,f)&\leq \log M(r,P_1)+\log M(r,P_2)+O(\log r)\\
    &\leq \int_0^r\frac{n(t,1/{P_1})+n(t,1/{P_2})}{t} dt\\
    &\quad+r\int_r^{\infty}\frac{n(t,1/{P_1})+n(t,1/{P_2})}{t^2}dt+O(\log r)\\
    &\leq N(r,1/f)+N(r,f)+O(\varphi(r)^{\sigma})+O(\log r),
    \end{split}
    \end{equation}  
    where $\sigma=\max\{\rho_\varphi(n(r,f)),\rho_\varphi(n(r,1/{f}))\}+\varepsilon$. 
By Remark~\ref{ilpo}(b) and \eqref{down} in Lemma~\ref{N-lemma}, we know $\rho_{\varphi}(n(r,f))\leq\rho_{\varphi}(f)-\beta_\varphi$, and similarly for $\rho_{\varphi}(n(r,1/f))$. Therefore, it follows from \eqref{SMT} that
    $$
    T(r,f)\leq N(r,1/f)+N(r,f)+O(\varphi(r)^{\rho_\varphi(f)-\beta_\varphi+\varepsilon})+O(\log r).
    $$   
 Moreover,  if $\rho_\varphi(f)-\mu_\varphi(f)<\beta_\varphi$ and $\beta_\varphi>0$, then 
 $$
 \varphi(r)^{\rho_\varphi(f)-\beta_\varphi+\varepsilon}=o(T(r,f)).
    $$
       Since $T\left(r,\frac{f-a}{f-b}\right)=T(r,f)+O(1)$, the assertion now follows.
\end{proof}

\begin{remark}
(a)\,Assuming that the limit $\lim_{r\to\infty}\frac{\varphi'(r)r}{\varphi(r)}$ exists, then (a) and (b) in Lemma~\ref{aux-lemma} hold for any given $\lambda_1>0$ and $\lambda_2>0$.

\vskip 2mm
(b)\,If $f$ is  entire in Lemma~\ref{aux-lemma}, then $\rho_\varphi(f)=\rho_\varphi(N(r,a,f))$, where $a\in\C$.

\vskip 2mm
(c)\,The condition $\beta_\varphi>0$ in Lemma~\ref{aux-lemma} implies a restriction on the growth of $\varphi(r)$. If $\rho_\varphi(f)<\infty$ is assumed, then $\rho(f)=0$. Indeed, if $\rho(f)>0$, then there exist $\varepsilon>0$ and an increasing sequence $(r_n)$ of positive real numbers tending to infinity such that 
    $$
    \frac{\log T(r_n,f)}{\log r_n}\geq \varepsilon>0,\quad n\in\N.
    $$
     On the other hand, the existence of \eqref{lim-gamma=beta} implies
     $$
     \frac{\log\log r}{\log\varphi(r)}\geq \frac{\beta_\varphi}{2}>0,\quad r\geq R_0.
     $$
Therefore,
     \begin{equation*}
     \begin{split}
     \rho_\varphi(f)&=\limsup_{r\to\infty}\frac{\log T(r,f)}{\log \varphi(r)}\geq \limsup_{r\to\infty}\frac{\beta_\varphi\log T(r,f)}{2\log\log r}   \\ 
     &\geq\limsup_{n\to\infty}\frac{\beta_\varphi\log T(r_n,f)}{2\log\log r_n}\geq\limsup_{n\to\infty}\frac{\varepsilon\beta_\varphi\log r_n}{2\log\log r_n}=\infty,
     \end{split}
     \end{equation*}
which is a contradiction.
Moreover, the error term in \eqref{4.6-1} might not be small compared to $T(r,f)$  if $\beta_\varphi=0$.
This happens, for example, in the case when $f(z)=e^z$, $a=0$, $b=\infty$ and $\varphi(r)=r$.
\end{remark}

We conclude this section with two examples, which show that if $\varphi(r)$ satisfies some specific conditions, then an entire function of
 $\varphi$-order satisfying \eqref{coro-rho} can be  constructed in terms of canonical products. Here, $\lambda$ denotes $\varphi$-exponent of convergence of a sequence $\{z_n\}$.

\begin{example}\label{positive-order-ex}
Suppose that $\varphi(r)$ is differentiable, strictly increasing and subadditive such that the limit \eqref{lim-gamma=beta}
exists and that 
    $
    \limsup_{r\to\infty}\frac{\varphi'(r)r}{\varphi(r)}<\frac{1}{\lambda}
    $ 
    holds for $\lambda=1/\kappa$, $0<\kappa<1$.  Then
there exists a $\tau>\kappa$ and an $R(\tau)>0$ such that
	$$
	\varphi(r)\leq r^{1/\tau},\quad r\geq R(\tau).
	$$
Then $\varphi^{-1}(r)\geq r^{\tau}$ for all $r\geq R(\tau)$. Define
    $$
    F(z)=\prod_{n=1}^\infty\left(1-\frac{z}{\varphi^{-1}(n^{1/\kappa})}\right).
    $$
We have
	$$
	\sum_{n=1}^\infty\frac{1}{\varphi^{-1}(n^{1/\kappa})}\leq \sum_{n^{1/\kappa}<R(\tau)}\frac{1}{\varphi^{-1}(n^{1/\kappa})}
	+\sum_{n^{1/\kappa}\geq R(\tau)}\frac{1}{n^{\tau/\kappa}}<\infty.
	$$
This guarantees that $F(z)$ is an entire function. Moreover, the $\varphi$-exponent
of convergence of the zeros of $F(z)$ is
    $$
    \lambda_{\varphi}(F)=\inf\bigg\{\mu>0: \sum_n\frac{1}{(n^{1/\kappa})^\mu}<\infty\bigg\}=\kappa.
    $$
Hence,  Lemmas~\ref{n-lemma} and~\ref{aux-lemma}
yield $\rho_{\varphi}(F)=\rho_{\varphi}(N(r,1/F))=\kappa+\beta_\varphi$.
 \end{example}

 \begin{example}
Suppose that $\varphi(r)$ is differentiable, strictly increasing and subadditive such that the limit \eqref{lim-gamma=beta}
exists and that either     
    $
    \limsup_{r\to\infty}\frac{\varphi'(r)r}{\varphi(r)}<\frac{1}{\lambda}
    $
 or
    $
    \liminf_{r\to\infty}\frac{\varphi'(r)r}{\varphi(r)}\geq\frac{1}{\lambda}
    $ 
     holds. For a given $c>1$, define
    $$
    G(z)=\prod_{n=1}^\infty\left(1-\frac{z}{\varphi^{-1}(c^n)}\right).
    $$
From \eqref{general-restriction} we find that $r\leq \varphi^{-1}(r)\leq e^r$ for all $r\geq R_0$.
Similarly as in Example~\ref{positive-order-ex}, we see that $G(z)$ is entire, and that
    $$
    \lambda_{\varphi}(G)=\inf\bigg\{\mu>0: \sum_n\frac{1}{(c^{n})^\mu}<\infty\bigg\}=0.
    $$
Moreover, $\rho_{\varphi}(G)=\rho_{\varphi}(N(r,1/G))=\beta_\varphi$.
\end{example}

%%%%%%%%%%%%%%%%%%%%%%%%%%%%%%%%%%%%%%%%%%%%%%
% SUBSECTION 5
%%%%%%%%%%%%%%%%%%%%%%%%%%%%%%%%%%%%%%%%%%%%%%

\section{Proof of Theorem \ref{growth-coe}}\label{proof of Th4.1}

(a)\,Trivially we may suppose that $\alpha_{\varphi,s}>0$ and $\rho_\varphi(f)<\infty$. Moreover, we suppose first that the coefficients $a_0(z),\ldots,a_n(z)$ are entire.
Let $\varepsilon\in (0,\alpha_{\varphi,s})$ be small such that $\rho_\varphi(a_i)\geq\displaystyle\max_{j\neq i}\{\rho_\varphi(a_j)\}+2\varepsilon$.
From the definitions of $\alpha_{\varphi,s}$ and $\gamma_{\varphi,s}$, we have
	$$
	\varphi(r)\geq \varphi(s(r))^{\alpha_{\varphi,s}-\varepsilon^2}
	\quad\textnormal{and}\quad
	\log \frac{s(r)}{r}\geq \varphi(r)^{\gamma_{\varphi,s}-\frac{\varepsilon}{3\alpha_{\varphi,s}}}.
	$$
for all $r\geq R_0$. We divide \eqref{diffeqn} by $f(q^iz)$ and use Lemma~\ref{q-difference-lemma}(a) to get
	\begin{equation}\label{es-m-ai}
	\begin{split}
	m(r,a_i)\leq& n\max_{j\neq i}\{m(r,a_j)\}
	+O\left(\frac{\varphi(s(r))^{\rho_\varphi(f(q^iz))+\frac{\varepsilon}{3}}}{\log\frac{s(r)}{r}}\right)\\
	\leq& O\left(\varphi(r)^{\rho_\varphi(a_i)-\veps}\right)
	+O\left(\frac{\varphi(r)^{\frac{1}{\alpha_{\varphi,s}}(\rho_\varphi(f(q^iz))
	+\frac{2\varepsilon}{3})}}{\log\frac{s(r)}{r}}\right)\\
	\leq& O\left(\varphi(r)^{\rho_\varphi(a_i)-\veps}\right)
	+O\left(\varphi(r)^{\frac{\rho_\varphi(f(q^iz))}{\alpha_{\varphi,s}}-\gamma_{\varphi,s}+\frac{\varepsilon}{\alpha_{\varphi,s}}}\right)
	\end{split}
	\end{equation}
for all $r\geq R_0$. We recall from \cite[p.~249]{BIY} that, for a meromorphic function $h(z)$,
 	$
	T(r,h(Cz))=T(|C|r,h(z))+O(1)
	$
	holds for every $C\in\C\setminus\{0\}$,
where the $O(1)$-term depends on $C$. Then the subadditivity of $\varphi(r)$ yields 
    \begin{equation}\label{BIY-C}
    \rho_\varphi(h(Cz))=\rho_\varphi(h(z)),\quad C\in\C\setminus\{0\}.
    \end{equation}
It follows from \eqref{es-m-ai} and \eqref{BIY-C} that
    \begin{equation}\label{m-rho-estimate-new}
    m(r,a_i)= O\left(\varphi(r)^{\rho_\varphi(a_i)-\veps}\right)
	+O\left(\varphi(r)^{\frac{\rho_\varphi(f)}{\alpha_{\varphi,s}}-\gamma_{\varphi,s}+\frac{\varepsilon}{\alpha_{\varphi,s}}}\right).
	\end{equation}
Since there exists a sequence $\{r_n\}$ of positive real numbers tending
to infinity such that $m(r_n,a_i)\geq \varphi(r_n)^{\rho_\varphi(a_i)-\frac{\varepsilon}{2}}$,
    $$
    \rho_{\varphi}(a_i)-\frac{\varepsilon}{2}\leq \frac{\rho_{\varphi}(f)}{\alpha_{\varphi,s}}-\gamma_{\varphi,s}
    +\frac{\varepsilon}{\alpha_{\varphi,s}}.
    $$
Since $\varepsilon>0$ is arbitrarily small, the assertion follows.

Suppose then that some of the coefficients $a_0(z),\ldots,a_n(z)$
have poles. We divide \eqref{diffeqn} again by $f(q^iz)$ and apply \eqref{BIY-C}  to get
    \begin{align*}
	N(r,a_i) &\leq n\max_{j\neq i}\{N(r,a_j)\}+O(T(r,f))+O(T(r,f(q^nz)))\\
	&\leq O\left(\varphi(r)^{\rho_\varphi(a_i)-\veps}\right)+O\left(\varphi(r)^{\rho_\varphi(f)+\varepsilon}\right).
	\end{align*}
 Combining this with \eqref{m-rho-estimate-new}, we have
    $$
    \rho_{\varphi}(a_i)-\frac{\varepsilon}{2}\leq
    \max\left\{\frac{\rho_{\varphi}(f)}{\alpha_{\varphi,s}}-\gamma_{\varphi,s},\rho_{\varphi}(f)\right\}+\frac{\varepsilon}{\alpha_{\varphi,s}},
    $$
where we may let $\varepsilon\to 0^+$.

(b)\,
Making use of Lemma~\ref{q-difference-lemma}(b) in  the proof of Case (a) above and following the same method,  it is easy to obtain the conclusion.

%%%%%%%%%%%%%%%%%%%%%%%%%%%%%%%%%%%%%%%%%%%%%%
% SUBSECTION 6
%%%%%%%%%%%%%%%%%%%%%%%%%%%%%%%%%%%%%%%%%%%%%%

\section{Proof of Theorem \ref{rho f-rho a}}\label{proof of Theorem 4.2}

We need the following lemma for the proof.

\begin{lemma}{\rm(\cite[Lemma~3.3]{Heittokangas})}\label{Hei}
Let the coefficients $a_0(z),\ldots,a_{n+1}(z)$ of \eqref{diffeqn-non-homo} be meromorphic. If $f$ is a meromorphic solution of \eqref{diffeqn-non-homo}, then there exists a finite constant $C>0$ such that
    $$
    n(r,f)\leq C\left(\sum_{j=0}^{n+1}n(r,a_j)+n\left(r,\frac{1}{a_0}\right)\right)\log r,\quad r\geq R_0.
    $$
\end{lemma}

We are now ready to give the proof of Theorem~\ref{rho f-rho a}.

It is proved in \cite[Theorem~1.1]{BIY} that any meromorphic solution $f$ of \eqref{diffeqn-non-homo} satisfies $T(r,f)=O(\log^2 r)$, when  $a_0,\ldots,a_{n+1}$ are rational. Thus \mbox{$\rho_\varphi(f)\leq 2\beta_\varphi$.} If at least one of  $a_0,\ldots,a_{n+1}$ is non-constant, it follows from Corollary~\ref{rho>alp.lamb-alp.gam} and $\alpha_{\varphi,s}>0$ that $\rho_\varphi/\alpha_{\varphi,s}-\gamma_{\varphi,s}\geq 0$, where $\rho_{\varphi}$ is defined in \eqref{rho-varphi-coe}. Then the assertion follows. If all of $a_0,\ldots,a_{n+1}$ are constants, then \cite[p.~249]{BIY} shows that any meromorphic solution $f$ of \eqref{diffeqn-non-homo} is rational, that is, $\rho_\varphi(f)\leq \beta_\varphi$.

Now we suppose that at least one of $a_0,\ldots,a_{n+1}$ is transcendental.
First, we estimate $N(r,f)$, which is the easy part of the proof. For a non-constant meromorphic function $h(z)$ of finite $\varphi$-order, it follows from Corollary~\ref{rho>alp.lamb-alp.gam}  that
    \begin{equation}\label{lambda-L}
    0\leq\lambda_\varphi\leq \frac{\rho_{\varphi}(h)}{\alpha_{\varphi,s}}-\gamma_{\varphi,s},\quad \alpha_{\varphi,s}\in(0,1],
    \end{equation}
   where $\lambda_{\varphi}$ is the $\varphi$-exponent of convergence of the
$a$-points of $h(z)$.
Then making use of \eqref{liminf}, \eqref{L}, \eqref{N geq n} and $\alpha_{\varphi,s}\in(0,1]$, we obtain
    \begin{equation}\label{n-h-up}
    \begin{split}
    n(r,h)&\leq \frac{N(s(r),h)}{\log\frac{s(r)}{r}}
    =O\left(\frac{\varphi(s(r))^{\rho_\varphi(h)+\frac{\varepsilon}{3}}}{\log\frac{s(r)}{r}}\right)\\
    &=O\left(\frac{\varphi(r)^{\frac{1}{\alpha_{\varphi,s}}(\rho_\varphi(h)+\frac{2\varepsilon}{3})}}{\log\frac{s(r)}{r}}\right)
    =O\left(\varphi(r)^{\frac{\rho_\varphi(h)}{\alpha_{\varphi,s}}-\gamma_{\varphi,s}+\frac{\varepsilon}{\alpha_{\varphi,s}}}\right) 
	\end{split}
    \end{equation}
 for all $r\geq R_0$.
Applying \eqref{n-h-up}  to any non-constant function among $1/a_0$ and $a_0,\ldots,a_{n+1}$, and then using \eqref{rho-varphi-coe},
\eqref{N leq n log r} and Lemma~\ref{Hei}, we have
    \begin{equation}\label{N-varphi order}
    N(r,f)\leq n(r,f)\log r+O(\log r)
    = O\left(\varphi(r)^{\frac{\rho_\varphi}{\alpha_{\varphi,s}}-\gamma_{\varphi,s}+\frac{\varepsilon}{\alpha_{\varphi,s}}}\log^2 r\right).
    \end{equation}

Second, we estimate $m(r,f)$, which is the laborious part of the proof. Assume first that the coefficients
$a_0,\ldots,a_{n+1}$ are entire and that $0<|q|<1$, and denote
$p=1/q$.
It follows from the Hadamard's factorization theorem~\cite[p.~22]{Boas} and the global assumption \eqref{global-rho<1} that  $a_0$ can be written as
    $
    a_0(z)=Cz^m P(z),
    $
       where $P(z)$ is the canonical product formed with the \mbox{non-zero} zeros of $a_0$, $C$ is a non-zero complex constant and $m$ is an integer.
If $a_0$ is a polynomial, then $\frac{1}{|a_0(z)|}=O(1)$, so we next consider the case that $a_0$ is transcendental.
Following the proof of \cite[Theorem~3.5]{Heittokangas} or \cite[Theorem~5.3]{ZT}, we apply Lemma~\ref{canonical} to conclude
    \begin{align*}
	\frac{1}{|a_0(z)|}=\frac{1}{|C|r^m|P(z)|}\leq\exp\left(\varphi(s(r))^{\lambda+\frac{\varepsilon}{2}}\log r\right),
	\quad |z|=r\geq R_0,
	\end{align*}
provided that $z$ lies outside of discs $D_j$ of radius $|z_j|^{-(\lambda+\varepsilon)}$ around the zeros $z_j$ of $a_0$,
and where $\lambda$ is the $\varphi$-exponent of convergence of $\{z_j\}$.
Then by using \eqref{liminf} and applying \eqref{lambda-L} to some $a_k$ of $\varphi$-order $\rho_\varphi$ for  $0\leq k\leq n+1$, we have  
    \begin{equation}\label{a0}
    \begin{split}
    \frac{1}{|a_0(z)|} &\leq \exp\left(\varphi(r)^{\frac{\lambda}{\alpha_{\varphi,s}}+\frac{\varepsilon}{\alpha_{\varphi,s}}}\log r\right)\\
    &\leq \exp\left(\varphi(r)^{J_0+\frac{\varepsilon}{\alpha_{\varphi,s}}}\log r\right),\quad |z|=r\geq R_0,
    \end{split}
    \end{equation}
where $z$ lies outside of the discs $D_j$ and
	$$
	J_0=\frac{\rho_{\varphi}}{\alpha_{\varphi,s}^2}-\frac{\gamma_{\varphi,s}}{\alpha_{\varphi,s}}\geq 0.
	$$

We proceed to prove that
    \begin{equation}\label{max-M-a-j}
    \max_{0\leq j\leq n+1}\{M(r,a_j)\}\leq \exp\left(\varphi(r)^{J_0+\varepsilon}\log r\right),\quad r\geq R_0.
    \end{equation}
Indeed, it follows from \eqref{varphi-logM} and \eqref{rho-varphi-coe} that
	\begin{equation}\label{step1}
	\max_{0\leq j\leq n+1}\{M(r,a_j)\} \leq
	\exp\left(\varphi(r)^{\rho_\varphi+\frac{\varepsilon}{2}}\right),\quad r\geq R_0.
	\end{equation}
From \eqref{L}, we get
	$\log\frac{s(r)}{r}\geq \varphi(r)^{\gamma_{\varphi,s}-\frac{\varepsilon}{2}}$ for all $r\geq R_0$,
	which in turn implies
    \begin{equation*}\label{liminf-definition-leq}
    \varphi(r)^{\frac{\rho_{\varphi}}{\alpha_{\varphi,s}}-\alpha_{\varphi,s} J_0-\frac{\varepsilon}{2}}
    \leq {\log\frac{s(r)}{r}}\leq \log r,\quad r\geq R_0,
    \end{equation*}
    provided that $s(r)\leq r^2$. Thus, it follows from $\alpha_{\varphi,s}\in(0,1]$ that
    $$
    \varphi(r)^{\rho_{\varphi}+\frac{\varepsilon}{2}}
    \leq \varphi(r)^{\alpha_{\varphi,s}^2J_0+\frac{(\alpha_{\varphi,s}+1)\varepsilon}{2}}\log^{\alpha_{\varphi,s}} r
    \leq \varphi(r)^{J_0+\varepsilon}\log r,
    $$
	which leads \eqref{max-M-a-j} by using \eqref{step1}.

There exists a $t\in(1,|p|)$ such that $f$ has no poles on the circles \mbox{$|z|=r=|p|^kt$ } for $k\in\N$ and, by Remark~\ref{radii-of-discs}(a),
these circles are outside of the discs $D_j$ of radius $|z_j|^{-(\lambda+\varepsilon)}$ centered at $z_j$. For $k\in \N\cup\{0\}$, define
    $$
    M_k=M(|p|^kt,f)+1.
    $$
Hence, we divide \eqref{diffeqn-non-homo} by $a_0$ to get
	$$
	|f(z)|\leq \frac{1}{|a_0(z)|}\left(\sum_{j=1}^n |a_j(z)||f(q^jz)|+|a_{n+1}(z)|\right).
	$$
Then by making use of \eqref{a0}, \eqref{max-M-a-j} and $\alpha_{\varphi,s}\in(0,1]$, we have
	\begin{align*}
    M(|p|^{k}t,f)&\leq M\left(|p|^kt,\frac{1}{a_0}\right)\left(\sum_{j=1}^{n}M(|p|^kt,a_j)M(|p|^{k-j}t,f)+M(|p|^kt,a_{n+1})\right)\\
    &\leq \exp\left(2\varphi(|p|^kt)^{J_0+\frac{\varepsilon}{\alpha_{\varphi,s}}}\log (|p|^kt)\right)
    \sum_{j=k-n}^{k-1}\left(M(|p|^{j}t,f)+1\right) \\
    &\leq n \exp\left(2\varphi(|p|^kt)^{J_0+\frac{\varepsilon}{\alpha_{\varphi,s}}}\log (|p|^kt)\right)M_{k-1}
    \end{align*}
    for $k\geq n$ large enough, say $k\geq k_0\geq n$. Thus
    $$
    M_{k} \leq (n+1)\exp\left(2\varphi(|p|^kt)^{J_0+\frac{\varepsilon}{\alpha_{\varphi,s}}}\log (|p|^kt)\right)M_{k-1},\quad k\geq k_0,
    $$
which implies
    $$
    \log  M_{k}\lesssim \varphi(|p|^kt)^{J_0+\frac{\varepsilon}{\alpha_{\varphi,s}}}\log (|p|^kt)  +\log M_{k-1},\quad k\geq k_0.
    $$
Inductively,
    \begin{equation*}
    \begin{split}
    \log  M_{k}&\lesssim k\varphi(|p|^kt)^{J_0+\frac{\varepsilon}{\alpha_{\varphi,s}}}\log (|p|^kt)  +\log M_{k_0}\\
    &\lesssim k\varphi(|p|^kt)^{J_0+\frac{\varepsilon}{\alpha_{\varphi,s}}}\log (|p|^kt),\quad k\geq k_0,
    \end{split}
    \end{equation*}
where we may write
	$$
	k=\frac{\log(|p|^kt)-\log t}{\log |p|}.
	$$
Therefore,
   \begin{equation}\label{m}
    m(|p|^{k}t,f)\leq \log M_k
    \lesssim \varphi(|p|^kt)^{J_0+\frac{\varepsilon}{\alpha_{\varphi,s}}}\log^2(|p|^kt),\quad k\geq k_0.
    \end{equation}

We now combine the estimates \eqref{N-varphi order} and \eqref{m}, and find that
    \begin{equation}\label{T-varphi-order}
    T(r,f)= O\left(\varphi(r)^{J_0+\frac{\varepsilon}{\alpha_{\varphi,s}}}\log^2 r\right)
    \end{equation}
holds for $r=|p|^kt$, $k\geq k_0$, where $J_0=\frac{\rho_{\varphi}}{\alpha_{\varphi,s}^2}-\frac{\gamma_{\varphi,s}}{\alpha_{\varphi,s}}\geq 0$ and $\alpha_{\varphi,s}\in(0,1]$.
Since $T(r,f)$ and $\varphi(r)^{J_0+\frac{\varepsilon}{\alpha_{\varphi,s}}}\log^2 r$ are increasing and $\varphi(r)$ is subadditive, there exists an integer $s\geq |p|$ such that for $r\in[|p|^kt,|p|^{k+1}t)$ and for $k\geq k_0$
     \begin{equation*}
     \begin{split}
      T(r,f)&\leq T(|p|^{k+1}t,f)\lesssim  \varphi(|p|^{k+1}t)^{J_0+\frac{\varepsilon}{\alpha_{\varphi,s}}}\log^2 (|p|^{k+1}t)\\
      &\lesssim \varphi(|p|^{k}t)^{J_0+\frac{\varepsilon}{\alpha_{\varphi,s}}}(2\log(|p|^{k}t))^2\lesssim  \varphi(r)^{J_0+\frac{\varepsilon}{\alpha_{\varphi,s}}}\log^2 r.
     \end{split}
     \end{equation*}
This implies that \eqref{T-varphi-order} holds for all $r$-values.
Thus,
    $$
    \rho_{\varphi}(f)\leq
    \frac{\rho_{\varphi}}{\alpha_{\varphi,s}^2}-\frac{\gamma_{\varphi,s}}{\alpha_{\varphi,s}}+\frac{\varepsilon}{\alpha_{\varphi,s}}+2\beta_\varphi.
    $$
Since $\varepsilon>0$ is arbitrary, the assertion \eqref{L-alpha-beta} is now proved in the case of entire coefficients for $0<|q|<1$.

If $|q|>1$, we appeal to a change of variable in \eqref{diffeqn-non-homo} by replacing $z$ with $s^nz$, where $s=1/q$. Then we proceed to consider the growth of meromorphic solutions of a  non-homogeneous $s$-difference equation of the form 
	\begin{equation*}
	\sum_{j=0}^na_{n-j}(s^nz)f(s^jz)=a_{n+1}(s^nz).
	\end{equation*}
 Applying \eqref{BIY-C} to $a_j(s^nz)$ for $0\leq j\leq n+1$, we have $\rho_\varphi(a_j(s^nz))=\rho_\varphi(a_j(z))$,  $0\leq j\leq n+1$. Hence,
     $
     \rho_\varphi=\max_{0\leq j\leq n+1}\{\rho_\varphi(a_j(s^nz))\}
     $ 
     follows from \eqref{rho-varphi-coe}.
We then make use of the proof of the case $0<|q|<1$ above, and the assertion \eqref{L-alpha-beta} follows in the case of entire coefficients for $|q|>1$.

Finally, suppose that the coefficients in \eqref{diffeqn-non-homo} are meromorphic. We claim that the coefficients
can be represented by quotients of entire functions of \mbox{$\varphi$-order} $\leq \rho_\varphi$.
Indeed, by the Miles-Rubel-Taylor theorem \cite[\S 14]{Rubel}, each $a_j(z)$ can be expressed in the form
    \begin{equation}\label{aj-mero}
    a_j(z)=\frac{P_j(z)}{Q_j(z)},\quad 0\leq j\leq n+1,
    \end{equation}
where $P_j(z)$, $Q_j(z)$ are entire functions and $A>0$ and $B>0$ are absolute constants such that, for all $r\geq R_0$,
    \begin{equation}\label{P-Q}
    T(r,P_j)\leq AT(Br,a_j)\quad\text{and}\quad T(r,Q_j)\leq AT(Br,a_j).
    \end{equation}
Let $m$ be the smallest integer such that $m\geq B$. Using the subadditivity of $\varphi(r)$, we have $\varphi(Br)\leq \varphi(mr)\leq m \varphi(r)$.
Then \eqref{aj-mero} and \eqref{P-Q} yield $\max\{\rho_{\varphi}(P_j),\rho_{\varphi}(Q_j)\}= \rho_{\varphi}(a_j)\leq \rho_{\varphi}$.

By multiplying away the denominators, \eqref{diffeqn-non-homo} can be re-written in the form
    	\begin{equation*}
	\sum_{j=0}^nb_j(z)f(q^jz)=b_{n+1}(z),
	\end{equation*}
where the new coefficients $b_0(z),\ldots,b_{n+1}(z)$ are entire functions such that
    $$
  	\widehat{\rho}_\varphi = \max_{0\leq j\leq n+1}\{\rho_\varphi(b_j)\}\leq \rho_{\varphi}.
    $$
Then we proceed similarly as above and conclude the assertion \eqref{L-alpha-beta} in the case of meromorphic coefficients.
This completes the proof.

%%%%%%%%%%%%%%%%%%%%%%%%%%%%%%%%%%%%%%%%%%%%%%
% Funding
%%%%%%%%%%%%%%%%%%%%%%%%%%%%%%%%%%%%%%%%%%%%%%
\section*{Funding}

The second author was supported by National Natural Science Foundation of China (No.~11771090).
The third author was supported by the National Natural Science Foundation of China (No.~11971288 and No.~11771090) and Shantou University SRFT (NTF18029).
The fourth author would like to thank the support of the China Scholarship Council (No.~201806330120).

\end{document}